\documentclass[11pt]{amsart}
\usepackage{amsmath}
\usepackage{amssymb, verbatim}
\usepackage{pstricks,pst-node,pst-plot}
\usepackage{comment}
\usepackage{color}
\usepackage{extarrows}

\title[]{Special Lagrangian cycles and Calabi-Yau transitions}
\author[T. C. Collins]{Tristan C. Collins}
  \email{tristanc@mit.edu}
  \address{Department of Mathematics, Massachusetts Institute of Technology, 77 Massachusetts Avenue, Cambridge, MA 02139}
 \thanks{T.C.C is supported in part by NSF CAREER grant DMS-194452 and an Alfred P. Sloan Fellowship. }

\author[S. Gukov]{Sergei Gukov}
  \email{gukov@theory.caltech.edu}
  \address{Walter Burke Institute for Theoretical Physics, California Institute of Technology, Pasadena, CA, 91125}
 \thanks{S.G.~is supported by the U.S.~Department of Energy, Office of Science, Office of High Energy Physics, under Award No.~DE-SC0011632, and by the National Science Foundation under Grant No.~NSF DMS 1664227. }
 
 \author[S. Picard]{Sebastien Picard}
  \email{spicard@math.ubc.ca}
  \address{Department of Mathematics, UBC, 1984 Mathematics Road,
    Vancouver, BC, Canada}
  \thanks{}
  
 \author[S.-T. Yau]{Shing-Tung Yau}
  \email{yau@math.harvard.edu}
  \address{Department of Mathematics, Harvard University, 1 Oxford St.,Cambridge, MA, 02138}
  \thanks{} 

\theoremstyle{plain}
\newtheorem{thm}{Theorem}[section]

\newtheorem{lem}[thm]{Lemma}

\theoremstyle{definition}
\newtheorem{ex}[thm]{Example}
\newtheorem{rk}[thm]{Remark}

\numberwithin{equation}{section}

\newcommand{\del}{\partial}

\newcommand{\dbar}{\overline{\del}}
\newcommand{\ddb}{\sqrt{-1}\del\dbar}

\newcommand{\F}{\mathcal{F}}
\newcommand{\LL}{\mathcal{L}}

\newcommand{\Q}{\mathcal{Q}}

\newcommand{\be}{\begin{equation}}
\newcommand{\bea}{\begin{eqnarray}}
\newcommand{\eea}{\end{eqnarray}} 
\newcommand{\ee}{\end{equation}}

\renewcommand{\leq}{\leqslant}
\renewcommand{\geq}{\geqslant}

\renewcommand{\epsilon}{\varepsilon}

\begin{document}
\maketitle

\begin{abstract}
We construct special Lagrangian 3-spheres in non-K\"ahler compact threefolds equipped with the Fu-Li-Yau geometry. These non-K\"ahler geometries emerge from topological transitions of compact Calabi-Yau threefolds. From this point of view, a conifold transition exchanges holomorphic 2-cycles for special Lagrangian 3-cycles.
\end{abstract}

\section{Introduction}
A broad goal in complex geometry and theoretical physics is to understand the parameter space of all simply connected Calabi-Yau threefolds. The possible Hodge diamonds of these manifolds are parametrized by two integers $h^{1,1}$ and $h^{2,1}$. A program of Candelas-Green-Hubsch, Clemens, Friedman and Reid \cite{Clemens,CGH,Friedman,Reid} proposes a procedure for traveling between two topologically distinct Calabi-Yau threefolds by birational contraction followed by smoothing. The simplest instance of such a topological transition is a conifold transition $Y \rightarrow \underline{Y} \rightsquigarrow X$. A conifold transition contracts 2-spheres in the initial threefold $Y$, which decreases $h^{1,1}$, and introduces new 3-spheres in the resulting threefold $X$, which increases $h^{2,1}$. It is conjectured \cite{CGH,Reid} that all simply connected Calabi-Yau threefolds can be joined by birational contraction followed by smoothing (see \cite{Rossi} for a survey of these ideas). There is also a symplectic mirror to this operation \cite{STY}.

\par As noted in \cite{Friedman,Reid}, collapsing sufficiently many 2-spheres during a conifold transition $Y \rightarrow \underline{Y} \rightsquigarrow X$ produces a complex manifold with second Betti number $b_2(X)=0$. Thus a conifold transition may connect a K\"ahler Calabi-Yau threefold to a non-K\"ahler complex manifold with trivial canonical bundle. Such examples starting from a quintic in $\mathbb{P}^4$ are given in \cite{Friedman} (see also \cite{LuTian} for more examples of going from K\"ahler to non-K\"ahler). 

\par This suggests that certain non-K\"ahler manifolds should be included in the parameter space of Calabi-Yau threefolds. The study of non-K\"ahler Calabi-Yau geometry in theoretical physics was initiated by Strominger \cite{Strominger}, and has since grown into an active area of research; see e.g. \cite{AGF,Fei,FHP,FeiYau,FIUV,FGV,FLY,FY,HIS,OUV} for examples, see \cite{AGS,BBDG,BBFTY,CdlOMc,dlOSvanes,GMW,Grana,Hull} for developments in string theory, and see \cite{FPPZA,FPPZB,GJS,GRT,GRST,LTY,LiuKFYang,Phong,PPZ18a,PPZ18b,PPZ18c,ST,STW,Tosatti,TsengYau} for research programs in this area.

\par We would like to understand how geometric objects evolve as they move across a topological transition of Calabi-Yau threefolds. We refer to \cite{CandelasdlOssa,GMS,AplusB,Stromingerholes} for related ideas in K\"ahler string theory. In our setup, our initial object is a K\"ahler Calabi-Yau threefold $Y$ equipped with a K\"ahler Ricci-flat metric \cite{Yau78}. A topological transition will take us to another, possibly non-K\"ahler, complex manifold $X$. We list below some geometric properties which are known to be preserved as we travel from $Y$ to $X$.

\par  Fu-Li-Yau \cite{FLY} showed that $X$ admits a balanced metric $\omega$, meaning that it satisfies $d \omega^2=0$. This condition appears in string theory as a condition for supersymmetry \cite{LY,Strominger}. Friedman \cite{FriedmanDDB} showed that $X$ satisfies the $\partial \bar{\partial}$-lemma. It was shown in \cite{CPY} that the stability of the tangent bundle is preserved by the transition, in the sense that the tangent bundle of $X$ admits a Hermitian-Yang-Mills connection with respect to the balanced metric $\omega$. The stability of other bundles satisfying a local triviality condition near the singular points was studied in \cite{Chuan}.

\par In this work, we will consider special submanifolds of Calabi-Yau threefolds. On one side of the transition, we consider holomorphic 2-cycles. On the other side, we consider special Lagrangian 3-cycles. Let $Y \rightarrow \underline{Y} \rightsquigarrow X_t$ be a conifold transition contracting holomorphic 2-spheres $C_i \subset Y$ and replacing them with vanishing 3-spheres $L_{i,t} \subset X_t$. To understand this topological change geometrically, Fu-Li-Yau \cite{FLY} (see also \cite{Chuan,CPY} for follow-up work) construct a sequence of metrics $\omega_a$ on $Y$ with $d \omega_a^2=0$ and ${\rm Vol}(C_i,\omega_a) \rightarrow 0$ as $a \rightarrow 0$, and $\omega_t$ on $X_t$ with $d \omega_t^2=0$ with ${\rm Vol}(L_{i,t},\omega_t) \rightarrow 0$ as $t \rightarrow 0$. We will perturb $L_{i,t}$ to a special submanifold with respect to this geometry.

\par Our theorem states that a conifold transition exchanges the holomorphic 2-cycles $C_i$ on $(Y,\omega_a,\Omega)$ for special Lagrangian 3-cycles $\tilde{L}_{i}$ on $(X_t,\omega_t,\Omega_t)$.

\begin{thm} \label{thm:3spheres}
  Let $Y$ be a compact K\"ahler simply connected Calabi-Yau threefold. Let $Y \rightarrow \underline{Y} \rightsquigarrow X_t$ be a conifold transition, where $X_t$ is a compact complex manifold (which could be non-K\"ahlerian) with holomorphic volume form $\Omega_t$. Choose a node $p \in \underline{Y}$ and let $L_t$ denote the corresponding vanishing cycle on $X_t$. 

\par There exists $\epsilon>0$ such that for all $0<|t|<\epsilon$, the vanishing cycle $L_t$ can be perturbed to a rigid special Lagrangian 3-sphere $S^3$ with respect to the Fu-Li-Yau $(X_t,\Omega_t,\omega_t)$ non-K\"ahler Calabi-Yau structure. Explicitly, we solve the equations
  \[
\omega_t|_{S^3} =0, \quad ({\rm Im} \, e^{-i \hat{\theta}} \Omega_t)|_{S^3}=0, \quad d \omega_t^2=0
\]
for an angle $e^{-i \hat{\theta}} \in S^1$.
\end{thm}

\par Since $X$ may be not admit any K\"ahler metric, our notion of special Lagrangian does not involve a symplectic structure. We use the definition of Harvey-Lawson \cite{HarveyLawson} \S V.3., where a submanifold $L$ of a complex manifold $X$ with hermitian metric $\omega$ and holomorphic volume form $\Omega$ is said to be special Lagrangian with angle $e^{-i \hat{\theta}} \in S^1$ if $\omega|_L=0$ and $({\rm Im} \, e^{-i \hat{\theta}} \Omega)|_L = 0$. These equations were obtained by Harvey-Lawson (and later rederived in string theory \cite{BBS}) as a condition which minimizes the functional
\[
E(L) = \int_L |\Omega|_\omega d {\rm vol}_L
\]
in a given homology class. When $\omega$ is K\"ahler Ricci-flat, then $|\Omega|_\omega$ is constant and this is the area functional, but in general the norm $|\Omega|_\omega$ may fluctuate. Given a non-K\"ahler structure $(\omega,\Omega)$, special Lagrangian cycles can be understood as optimal representatives of a 3-cycle homology class with respect to the functional $E(L)$.

\par Our construction establishes the existence of smooth special Lagrangian cycles with respect to a balanced metric on any compact threefold emerging from a conifold transition. By Friedman's criteria \cite{Friedman,Tian92}, to guarantee the existence of a transition it suffices to check a cohomological condition on holomorphic curves on the initial threefold (\ref{Friedmans-cond}). For a concrete non-K\"ahler example, we may apply our theorem to conclude the existence of special Lagrangian submanifolds in connected sums of $h$ copies of $S^3 \times S^3$ with $h \geq 2$ \cite{Friedman,LuTian}.

\par For related results on special Lagrangian vanishing spheres and smoothings $\underline{Y} \rightsquigarrow X_t$ of nodal points in the case when $X_t$ is K\"ahler, see \cite{Chan,HeinSun}. Other constructions of special Lagrangian cycles in the K\"ahler setting can be found in e.g. \cite{CJL,CJL2,Joyce3spheres,JoyceIII,KarigiannisOo,YangLiSYZ,YangLiSYZ2,SchoenWolf,YuguanZhang}.

\par Following Callan-Harvey-Strominger \cite{CHS}, we can interpret our results as an analog of the Dirac magnetic monopole in electrodynamics. The magnetic monopole equations are
\[
\quad d^* F = 0, \quad Q_M = \int_{S^2} F
\]
where $F$ is the 2-form field strength and $Q_M$ is the magnetic charge enclosed in a 2-sphere. In string theory, there is a 3-form field strength $H$, and Strominger \cite{Strominger} found that on a Calabi-Yau threefold $X$ it is given by $H = i (\bar{\partial}-\partial) \omega$. If $\omega$ is a balanced metric on $X$, then $d \omega$ is primitive, and a well-known identity for the Hodge star gives $\star d \omega = J d \omega = -H$. Thus $H$ is co-closed, and the charge is defined analogously.
\[
d^* H = 0, \quad Q = \int_{S^3} H.
\]
In the case of smoothings $\underline{Y} \rightsquigarrow X_t$ with Fu-Li-Yau metric, our work gives estimates on the order of the charge attached to the degenerating special Lagrangian 3-spheres.
\[
  Q_t = \int_{S^3} H_t = O(|t|^{4/3}), \quad t \rightarrow 0.
\]
This can be derived from the estimates (\ref{eq:cot2check}) and (\ref{eq:U-defn}), which are used in the proof of our main theorem.

\par The paper is organised as follows. In Section~\ref{sec: BasicTheory}, we review generalities of special Lagrangian submanifolds in the non-K\"ahler setting and work out the equations for first order infinitesimal deformations of special Lagrangians. In Section~\ref{sec: Ex} we study some explicit examples of the deformation theory for special Lagrangians which show that the deformations behave rather differently than in the K\"ahler setting.  In Section~\ref{sec: backConifold}, we review the geometry of conifold transitions. In Section~\ref{sec: spheres}, we prove Theorem \ref{thm:3spheres} by a glueing and perturbation technique. The main idea is that the local K\"ahler Calabi-Yau model geometry $(V_t,\omega_{co,t},\Omega_{mod,t})$ of Candelas-de la Ossa \cite{CandelasdlOssa} on $V_t = \{ \sum_i z_i^2=t \}$ approximately describes the non-K\"ahler geometry of $(X_t,\omega_t,\Omega_t)$ near the vanishing cycles.  Finally, in Section~\ref{sec: phys} we discuss the relation to $SU(3)$ structures and flux compactifications in physics.
\bigskip

{\bf Acknowledgements:} We thank M. Garcia-Fernandez for helpful comments.

\section{Special Lagrangians without the K\"ahler condition}\label{sec: BasicTheory}
\subsection{Definitions} \label{section:defn}
We take the definition of a non-K\"ahler Calabi-Yau manifold as a complex manifold with trivial canonical bundle. Let $M$ be a compact complex manifold of complex dimension $n$, with complex structure $J$, admitting a nowhere vanishing holomorphic $(n,0)$ form $\Omega$. Let $g$ be a hermitian metric on $M$ with associated $(1,1)$ form $\omega(X,Y)=g(JX,Y)$. We define the function $| \Omega |_g$ by
\[
| \Omega |^2_g \, {\omega^n \over n!} =  {i^{n^2} \over 2^n} \Omega \wedge \bar{\Omega}.
\]
Harvey and Lawson \cite{HarveyLawson} introduced the definition of special Lagrangian cycles as special submanifolds of complex manifolds with trivial canonical bundle. An oriented submanifold $L$ of real dimension $n$ is special Lagrangian with angle $e^{-i \theta} \in S^1$ if 
\be \label{slag-defn}
\omega|_L =0, \quad {\rm Im} \, e^{-i \theta} \Omega|_L=0.
\ee
This equation was also derived by Becker-Becker-Strominger \cite{BBS} as a condition for $L$ to be a supersymmetric cycle. Though most subsequent work has been in the context of K\"ahler geometry \cite{BBS,GHJ,Hitchin,SYZ}, we return here to the general definition \cite{HarveyLawson} with possibly non-K\"ahler Hermitian metric $\omega$.
\smallskip
\par Submanifolds solving (\ref{slag-defn}) have the property that they minimize the functional
\be \label{conformal-vol}
E(L) = \int_L | \Omega |_g \, d {\rm vol}_L
\ee
in the homology class $[L] \in H_3(M,\mathbb{R})$. Here $d {\rm vol}_L$ is the volume form of the induced metric $g|_L$ on $L$. By a well-known pointwise computation \cite{HarveyLawson}, the condition (\ref{slag-defn}) is equivalent to an orientation on $L$ with $\omega|_L=0$, $e^{-i \theta} \Omega|_L = | \Omega |_g d {\rm vol}_L$.

\begin{rk}
  We can conformally change $\tilde{\omega}= | \Omega |^{2/n}_g \omega$ such that $| \Omega |_{\tilde{g}} = 1$. The special Lagrangian condition $\tilde{\omega}|_L=0$, ${\rm Im} \, e^{-i \theta} \Omega|_L=0$ is well-known \cite{HarveyLawson, JoyceIII} to minimize the area $\int_L d {\rm vol}_{\tilde{g}|_L}$ in a given homology class, using the closed form ${\rm Re} \, e^{-i \theta} \Omega$ as a calibration on $(M,\tilde{g})$. The discussion above rewrites this statement in terms of the original $\omega$ rather than $\tilde{\omega}$.
\end{rk}

When $d \omega = 0$ and $| \Omega |_g=1$, this is the standard special Lagrangian condition in a K\"ahler Calabi-Yau manifold. In the current work, we will not assume that $\omega$ is K\"ahler Ricci-flat, but instead require the weaker condition
\be \label{eq:confbal}
d (| \Omega |_g \omega^{n-1})=0
\ee
which is a condition for supersymmetry \cite{Strominger,LY}. In analogy with (\ref{conformal-vol}), it was observed in \cite{GRST} that metrics solving (\ref{eq:confbal}) are critical points of the functional
\[
\mathcal{M}(\omega) = \int_M |\Omega|_g \, d {\rm vol}_g
\]
when restricted to variations of the form $\delta \omega \in  {\rm Range} \, ( \partial \oplus \bar{\partial})$. Thus the special Lagrangian equations (\ref{slag-defn}) and the conformally balanced equation (\ref{eq:confbal}) are linked by the fact that they are both critical points of a functional of the form $\int |\Omega| \, d {\rm vol}$.

\par Metrics solving $d \omega^{n-1}=0$, called balanced metrics, were studied by Michelsohn \cite{Michelsohn} as a dual to the K\"ahler condition.  Metrics satisfying (\ref{eq:confbal}) are called conformally balanced. One can pass between these two notions by conformal change of the metric; a balanced metric $\chi$ defines a conformally balanced metric $\omega$ by $\omega = |\Omega|_\chi^{-2/(n-2)} \chi$, and a conformally balanced metric $\omega$ defines a balanced metric $\chi$ by $\chi= | \Omega |_\omega^{1/(n-1)} \omega$.

\subsection{Linearization} \label{section:linearization}
Next, we compute the linearization of the special Lagrangian condition. We will suppose that we have a family $L_s$ of special Lagrangians on a (possibly non-K\"ahler) Calabi-Yau manifold and study its first order variation. This calculation is well-known in the K\"ahler case, e.g. \cite{Marshall,McLean}.
\smallskip
\par Let $f: L \rightarrow M$ be a $n$-dimensional real submanifold of $(M,g,J,\omega,\Omega)$. Denote by $h=f^* g$ the induced metric on $L$ and $d {\rm vol}|_L$ the volume form of $h$. Suppose $L$ is a special Lagrangian with angle zero: $f^* \omega=0$ and $f^* ({\rm Im} \, \Omega)=0$ and $L$ is oriented so that $f^*{\rm Re} \, \Omega = | \Omega | d {\rm vol}_L$. Let $\alpha = \alpha_i du^i$ be a 1-form on $L$. Let $V= V^i {\partial \over \partial u^i}$ with $V^i = h^{ik} \alpha_k$ be the corresponding vector field on $L$, which on $M$ becomes $f_* V = V^i \partial_i f^\alpha \partial_\alpha$. 
\smallskip
\par The vector field $\xi:= J f_* V$ defined on $f(L) \subset M$ is a normal vector field since $f^* \omega=0$. Extend $\xi$ arbitrarily to all of $M$. Let $\varphi_s$ be a 1-parameter family of diffeomorphisms of $M$ generated by the flow of $\xi$. Let
\[
\F_s : \Lambda^1(L,\mathbb{R}) \rightarrow \Lambda^2(L,\mathbb{R}) \times C^\infty(L,\mathbb{R})
\]
be given by
\[
\F_s(\alpha) = \Big( -(\varphi_s \circ f)^* \omega, \, - \star_h (\varphi_s \circ f)^*{\rm Im } \, \Omega \Big).
\]
The condition that $\varphi_s$ deforms $L$ to a family of special Lagrangian submanifolds is $\F_s(\alpha) = 0$ for all $s \in (-\epsilon,\epsilon)$.
\smallskip
\par Differentiating $\F_s$ introduces the Lie derivative
\[
{d \over ds} \bigg|_{s=0} \F_s = \Big( -f^* (L_{\xi} \omega), \, -\star_h f^* (L_{\xi}{\rm Im} \, \Omega) \Big) .
\]
We will now compute the two terms on the right-hand side.
\smallskip
\par $\bullet$ $-f^*(L_\xi \omega)$. By the formula for the Lie derivative, this is
\[
-f^*(L_\xi \omega)=  - f^* (d \iota_{J f_*V} \omega + \iota_{J f_*V} d \omega).
\]
For the first term, tracing the definitions gives
\[
\alpha = - f^* (\iota_{J f_* V} \omega).
\]
For the second term, we write
\[
(f^* \iota_{J f_* V} d \omega)_{i j} = V^k (f^* \iota_J d \omega)_{k i j},
\]
where $(\iota_J d\omega)(X,Y,Z) = d \omega(JX,Y,Z)$. Therefore, we have
  \be \label{eq:LieDeriv1}
-f^* (L_{\xi} \omega) = d \alpha + T \alpha,
  \ee
  where $T: \Lambda^1(L) \rightarrow \Lambda^2(L)$ is defined by
\be \label{torsion-operator}
(T \alpha)_{i j} = - h^{k \ell} \alpha_\ell (f^* \iota_J d \omega)_{k i j}.
\ee
The $T \alpha$ terms are the non-K\"ahler contributions to the linearized operator.
\smallskip
\par $\bullet$ $-\star_h f^* (L_\xi {\rm Im} \, \Omega)$. Since $d \Omega = 0$, this is
\[
f^* (L_\xi {\rm Im} \, \Omega) = d f^* (\iota_\xi {\rm Im} \, \Omega).
\]
A well-known pointwise computation (e.g. \cite{Marshall}) gives
\be \label{eq:xi-Omega}
f^* (\iota_\xi {\rm Im} \Omega) =  \star_h (f^*| \Omega |_g \alpha).
\ee
To verify this, we start by choosing $p \in L$. Using $\omega|_L=0$ we may choose coordinates $(x_i,y_i)$ on $M$ so that
\[
T_pL = {\rm span} \bigg( {\partial \over \partial x_1} , \dots , {\partial \over \partial x_n} \bigg)
  \]
\[
g|_p = dx_i \otimes dx_i + dy_i \otimes dy_i, \quad J {\partial \over \partial x_i} = {\partial \over \partial y_i}, \quad J {\partial \over \partial y_i}= - {\partial \over \partial x_i}.
\]
Using ${\rm Im} \, \Omega|_L = 0$, we can arrange that 
\[
\Omega|_p = | \Omega |_g (dx_1+idy_1) \wedge \cdots \wedge (dx_n + i dy_n).
\]
We take $\alpha|_p = dx_1$ and can now verify (\ref{eq:xi-Omega}) by direct computation. Thus
    \[
      \star_h f^* (L_\xi {\rm Im} \, \Omega) = \star_h d \star_h (f^*| \Omega |_g \alpha) = - d^\dagger_h (f^* | \Omega |_g \alpha).
    \]
\par $\bullet$ Altogether, the variational formula  becomes
\be \label{eq:variation-formula}
\Big( -f^* (L_{\xi} \omega), \, - \star_h f^* (L_{\xi}{\rm Im} \, \Omega) \Big)= \bigg (d \alpha + T \alpha, \, d^\dagger_h (f^*| \Omega |_g \alpha) \bigg) .
\ee
The right-hand side is zero if and only if $\alpha$ satisfies $\LL \alpha = 0$, with
\be
\LL \alpha = | \Omega |^{-1} d_T^\dagger ( | \Omega | d_T \alpha) + d ( | \Omega |^{-1} d^\dagger ( | \Omega | \alpha)).
\ee
and $d_T = d + T$. Here $| \Omega |$ is notation for $f^* | \Omega |_g$ and adjoints $\dagger$ are with respect to the $L^2$ inner product on $(L,h)$. The operator $\LL$ is elliptic and can be deformed to the Hodge Laplacian, and so $\ker \LL$ is finite dimensional.  We state this as

\begin{lem}
The infinitesimal deformation space of a non-K\"ahler special Lagrangian $L\subset (X,g,\omega, \Omega)$ is given by the space of smooth $1$-forms $\alpha \in T^* L$ satisfying
\[
d\alpha + T\alpha =0 \qquad  d^\dagger (| \Omega |_\omega \alpha)=0
\]
where $T:T^*L \rightarrow \Lambda^2T^*L$ is defined by
\[
(T \alpha)_{i j} = - (g|_L)^{k \ell} \alpha_\ell (\iota_J d \omega)_{k i j}.
\]
Furthermore, the infinitesimal deformation space is finite dimensional.
\end{lem}

\begin{rk}
We note that that in the symplectic case $d \omega =0$, the constraint equations are $d\alpha=0$ and $d^\dagger (| \Omega | \alpha)=0$ and we recover the well-known fact 
\[
\dim \ker \LL = \dim H^1(L,\mathbb{R}).
\]
When $| \Omega |$ is constant, this is McLean's theorem \cite{McLean}.  In the case of non-constant $| \Omega |$, Goldstein \cite{Goldstein} showed that the map $\xi: \ker \LL \rightarrow H^1(L,\mathbb{R})$ given by $\alpha \mapsto [\alpha]$ is an isomorphism. This map is injective since if $\alpha=df$, then $d^\dagger(| \Omega | df)=0$ implies $f$ is constant. To show surjectivity, given $[\alpha] \in H^1(L,\mathbb{R})$ we must find a representative $\alpha' = \alpha + df$ such that $d^\dagger(| \Omega | \alpha')=0$, which amounts to solving
\[
  P(f) = \psi, \quad P(f) =  d^\dagger(| \Omega | df)
\]
with 
\[
\psi = - d^\dagger( | \Omega | \alpha).
\]
The operator $P$ is elliptic and $\ker P = \langle \mathbb{R} \rangle$. Therefore ${\rm Range}(P) = (\ker P)^\perp$ is given by functions $f$ such that $\int_L f = 0$. Since $\int_L \psi = 0$, we can solve $P(f) = \psi$.
\end{rk}

\begin{rk}
It is unclear whether the deformation theory of special Lagrangians in the non-K\"ahler case is unobstructed; that is, whether a version of McLean's theorem holds \cite{McLean}.  Indeed, the examples in the next section have unobstructed deformation theory.
\end{rk}

\section{Examples: Manifolds of Iwasawa-type}\label{sec: Ex}

In this section we work out an explicit example of the deformation theory for certain special Lagrangians in some non-K\"ahler Calabi-Yau manifolds.  One classical example of a compact complex manifold which is Calabi-Yau, but non-K\"ahler is the Iwasawa manifold, which can be defined as follows.  Consider the set of matrices
\[
G = \begin{pmatrix} 1 & z_1 & z_3\\ 0&1&z_2\\0&0&1 \end{pmatrix}.
\]
The Iwasawa manifold is given by $X_I:= G/\Gamma$ where $\Gamma \subset G$ is the subgroup consisted of those elements where $z_i \in \mathbb{Z}[\sqrt{-1}]$ for $i=1,2,3$.  It is straightforward to check that $X_I$ is complex, Calabi-Yau but non-K\"ahler; see, e.g. \cite[Chapter 3.5]{GriffithsHarris}.  There is a holomorphic fibration structure
\[
\pi:X_I \rightarrow T^4 \qquad \pi(z_1,z_2,z_3) = (z_1,z_2)
\]
whose fibers are complex tori.  The holomorphic $(1,0)$ forms 
\[
dz_1,\quad dz_2, \quad \theta_I:= dz_3-z_1dz_2
\]
are globally well-defined and trivialize $\Lambda^{1,0}T^*X_{I}$.  If ${u(z_1,z_2): T^4 \rightarrow \mathbb{R}}$ is any smooth, real function, then we can consider the hermitian form
\[
\omega_u:= \pi^*(e^{u}\omega_{T^4}) + \frac{\sqrt{-1}}{2}\theta_I\wedge \bar{\theta}_I
\]
where
\[
\omega_{T^4} = \frac{\sqrt{-1}}{2}\left(dz_1\wedge d\bar{z}_1 + dz_2\wedge d\bar{z}_2\right).
\]
It is easy to check that $\omega_u$ defines a hermitian metric, and furthermore $d\left(|\Omega|_{\omega_u} \omega_u^2\right) =0$, hence these metrics are conformally balanced.  Consider the anti-holomorphic involution $(z_1,z_2,z_3)\mapsto (\bar{z}_1,\bar{z}_2,\bar{z}_3)$.  The fixed point set is a special Lagrangian
\[
L = \{ {\rm Im}(z_1)= {\rm Im}(z_2) = {\rm Im}(z_3)=0\}
\]
In fact, it is not hard to check that $L$ moves in a $1$-parameter family given by
\[
L_{t}:= \{ {\rm Im}(z_1)= {\rm Im}(z_2) =0, \, {\rm Im}(z_3)=t\}.
\]
In this case, the main result of this section is
\begin{lem}\label{lem: Iwasawa}
Let $(X_I, \omega_u)$ be the Iwasawa manifold, as above.  The infinitesimal deformation space of the special Lagrangian $L$ is $1$-dimensional, while $b_1(L)=2>1$.  In particular, the infinitesimal deformations of the special Lagrangian $L$ are unobstructed.
\end{lem}

\begin{rk}
We have emphasized that $b_1(L)=2 >1$ to make it clear that the deformation theory for special Lagrangians in non-K\"ahler Calabi-Yau manifolds is very different from the deformation theory in K\"ahler Calabi-Yau manifolds.
\end{rk}

We will explain the proof of Lemma~\ref{lem: Iwasawa} at the end of this section.  We also consider a related example, given by $X:= \mathbb{C}^3/\sim$ where we defined
\[
(z_1,z_2,z_3) \sim (z_1+a, z_2+c, z_3+\bar{a}z_2+b)
\]
where $a,b,c \in \mathbb{Z}[\sqrt{-1}]$.  Note that the construction of $X$ differs from the construction of the Iwasawa manifold only in the introduction of the complex conjugate of $a$ in the action on $z_3$.  Nevertheless, $X$ is a smooth, complex, Calabi-Yau manifold and non-K\"ahler \cite{PicNotes}.  As in the case of the Iwasawa manifold there is a holomorphic fibration structure
\[
\pi:X \rightarrow T^4 \qquad \pi(z_1,z_2,z_3) = (z_1,z_2)
\]
whose fibers are complex tori.  There are non-vanishing $(1,0)$ forms 
\[
dz_1,\quad dz_2, \quad \theta:= dz_3-\bar{z}_1dz_2
\]
which are globally well-defined and trivialize $\Lambda^{1,0}T^*X$.  However, note that $d\theta \in \Lambda^{1,1}T^*X$, unlike the case of the Iwasawa manifold.  As above, if ${u(z_1,z_2): T^4 \rightarrow \mathbb{R}}$ is any smooth, real function, then we can consider the hermitian form
\[
\omega_u:= \pi^*(e^{u}\omega_{T^4}) + \frac{\sqrt{-1}}{2}\theta \wedge \bar{\theta}
\]
which are again conformally balanced. As before, there is a one-parameter family of special Lagrangians 
\[
L_{t}:=  \{ {\rm Im}(z_1)= {\rm Im}(z_2) =0, \, {\rm Im}(z_3)=t\}
\]
and it is the deformation theory of $L:=L_0$ that we will work out.  Before stating the main result, we point out that the primary interest of this example is that it serves as a local toy model for the manifolds constructed by Calabi-Eckmann \cite{CalEck}, generalized by Goldstein-Prokushkin \cite{GP} and exploited by Fu-Yau \cite{FY} in their construction of solutions of the Strominger system.  The main result is

\begin{lem}\label{lem: Fu-Yau}
Let $(X, \omega_u)$ be the compact Calabi-Yau manifold constructed above.  The infinitesimal deformation space of the special Lagrangian $L$ is $1$-dimensional, while $b_1(L)=2>1$.  In particular, the infinitesimal deformations of the special Lagrangian $L$ are unobstructed.
\end{lem}

\begin{proof}[Proof of Lemma \ref{lem: Fu-Yau}]
If we write $z_k = x_k + i y_k$, consider
\[
L_t = \{y_1=y_2=0, \, y_3=t \}.
\]
and let $L=L_0$.  One can check that, for each $t$, $L_t$ defines a smooth, closed submanifold of $X$, while direct computation gives
\[
\omega|_L=0, \quad \Omega|_L = | \Omega |_\omega d {\rm vol}_L.
\]

The goal is to try to understand the deformation theory of the special Lagrangian $L$.  We need to compute the operator
\[
T\alpha = -\alpha_kg^{k\ell}(\iota_{J}d\omega)_{\ell i j}
\]
on $L$.  The metric $g_u$ associated to the form $\omega_u= g_u(J\cdot, \cdot)$ is
\[
g_u =e^{u}\sum_{i=1}^{2}(dx_i^2+dy_i^2) + {\rm Re}(\theta \otimes \bar{\theta})
\]
and so
\[
g_{u}\big|_{L} = e^{u}(dx_1^2+dx_2^2) + (dx_3-x_1dx_2)^2.
\]
Note that $T^*L$ has a global frame given by the $1$-forms
\[
dx_1, \quad dx_2,\quad (dx_3-x_1dx_2).
\]
and since $u=u(x_1,x_2)$ is independent of $x_3$, each of the above $1$-forms is co-closed, as can be seen from the equations
\begin{equation}\label{eq: hodgeStarFrame}
\begin{aligned}
 \star dx_1&= dx_2\wedge (dx_3-x_1dx_2) \\
  \star dx_2&= -dx_1\wedge (dx_3-x_1dx_2) \\
  \star (dx_3-x_1dx_2) &= e^{u}dx_1\wedge dx_2.
  \end{aligned}
  \end{equation}
Given a $1$-form $\alpha$ we can write it as
 \begin{equation}\label{eq: alphaGlobTriv}
 \alpha = \alpha_1dx_1+\alpha_2dx_2 + \alpha_3(dx_3-x_1dx_2).
 \end{equation}
 The $g_u$-dual vector field is given by
 \[
 \alpha^{\#} = e^{-u}\alpha_1\frac{\del}{\del x_1} + e^{-u}\alpha_2\left(\frac{\del}{\del x_2} + x_1 \frac{\del}{\del x_3}\right) +\alpha_3\frac{\del}{\del x_3}
 \]
 and so
 \[
 J\alpha^{\#} = e^{-u}\alpha_1\frac{\del}{\del y_1} + e^{-u}\alpha_2\left(\frac{\del}{\del y_2} + x_1 \frac{\del}{\del y_3}\right) +\alpha_3\frac{\del}{\del y_3}.
 \]
By definition we have
 \[
 T\alpha = -d\omega(J\alpha^\#, \cdot, \cdot),
 \]
 and so we only need to compute $d\omega$.  In fact, it clearly suffices to compute the components of $d\omega$ of the form $dy_i\wedge dx_j\wedge dx_k$, since the remaining components will not contribute to $T\alpha$. In the following computation we will write ``Irr." to denote the irrelevant components of $d\omega$-- namely, those not contained in the span of the $3$-forms $dy_i\wedge dx_j\wedge dx_k$.  We have
 \[
 d\omega = e^{u}du\wedge \omega_{T^{4}} + {\rm Re}(\sqrt{-1}d\theta \wedge \bar{\theta})
 \]
 The first term is easier.  We have
 \begin{equation}\label{eq: TaOmComp}
e^{u}du\wedge \omega_{T^{4}} = e^{u} \frac{\del u}{\del x_1} dy_2\wedge dx_1\wedge dx_2- e^{u} \frac{\del u}{\del x_2} dy_1\wedge dx_1\wedge dx_2 + \left(\text{ Irr.}\right).
 \end{equation}
 For the second term we have
 \begin{equation}\label{eq: expdtheta}
 \begin{aligned}
 \sqrt{-1}d\theta \wedge\bar{ \theta} &= -d\bar{z_1}\wedge dz_2 \wedge (d\bar{\zeta} -z_1d\bar{z_2}) \\
 &= -\sqrt{-1}d\bar{z_1} \wedge dz_2 \wedge d\bar{\zeta} +z_1 d\bar{z_1}\wedge (\sqrt{-1}dz_2\wedge d\bar{z_2})
 \end{aligned}
 \end{equation}.
 For the second term on the right hand side of~\eqref{eq: expdtheta} we have
 \[
 z_1 d\bar{z_1}\wedge (\sqrt{-1}dz_2\wedge d\bar{z_2}) = 2x_1dy_2\wedge dx_1\wedge dx_2 + \left( \text{ Irr. }\right)
 \]
 while for the first term, we get
 \[
  -\sqrt{-1}d\bar{z_1} \wedge dz_2 \wedge d\bar{\zeta} = -(dy_1\wedge dx_2\wedge dx_3+dy_2\wedge dx_1\wedge dx_3+dy_3\wedge dx_1\wedge dx_2) + \left( \text{ Irr. } \right).
  \]
  All together, this yields
  \begin{equation}\label{eq: TaThetaComp}
  \begin{aligned}
 {\rm Re}\left(  \sqrt{-1}d\theta \wedge\bar{ \theta}\right) &= 2x_1dy_2\wedge dx_1\wedge dx_2 -\left(dy_1\wedge dx_2\wedge dx_3+dy_2\wedge dx_1\wedge dx_3\right)\\
 &\quad -dy_3\wedge dx_1\wedge dx_2) + \left( \text{ Irr. } \right)\\
 &= x_1dy_2\wedge dx_1\wedge dx_2 -dy_1\wedge dx_2\wedge (dx_3-x_1dx_2)\\
 &\quad -dy_2\wedge dx_1\wedge (dx_3-x_1dx_2)\\
 &\quad -dy_3\wedge dx_1\wedge dx_2 + \left( \text{ Irr. } \right).
 \end{aligned}
 \end{equation}
 Combining~\eqref{eq: TaOmComp} and~\eqref{eq: TaThetaComp} we obtain
 \begin{equation}\label{eq: TaFinal}
 \begin{aligned}
 T\alpha &= \left(\alpha_1\frac{\del u}{\del x_2}-\alpha_2\frac{\del u}{\del x_1} \right) dx_1\wedge dx_2\\
 &+\alpha_3dx_1\wedge dx_2+e^{-u}\alpha_1 dx_2\wedge (dx_3-x_1dx_2) + e^{-u}\alpha_2dx_1\wedge (dx_3-x_1dx_2)
 \end{aligned}
 \end{equation}
Our goal is to analyze the space of $1$-forms $\alpha$ satisfying the infinitesimal deformation equations for $L$
\[
d^{\dagger}(|\Omega|_{g_u}\alpha) =0 \qquad d\alpha+T\alpha=0.
\]

Before beginning the analysis, we discuss the topology of $L$, which is a nontrivial $T^2$ bundle over $S^1$.  Namely, consider the map $\pi : L \rightarrow S^1$ given by $\pi(x_1,x_2,x_3) = x_2 \in S^1$, then the fiber of $L$ is given by $(x_1,x_3) \in \mathbb{R}^2$ with the equivalence relation
\[
(x_1,x_3) \sim (x_1+a,x_3+ax_2+b) \qquad (a,b) \in \mathbb{Z}^2
\]
This is evidently a $T^2$ bundle, where we identify
\[
\pi^{-1}(x_2) = \mathbb{R}^2/\left(\mathbb{Z}\tau_1(x_2)+ \mathbb{Z}\tau_2\right) \qquad \tau_1(x_2) = (1,x_2), \quad \tau_2= (1,0).
\]
 In particular, we see that the monodromy action $\pi_1(S^1)$ acts on $H_1(T^2)$ by
 \[
 \tau_1 \mapsto \tau_1+\tau_2, \quad \tau_2 \mapsto \tau_2.
 \]
 It follows from this calculation that $b_1(L)=2$. 

The next task is to expand the infinitesimal deformations equations in terms of the frame introduced in ~\eqref{eq: hodgeStarFrame}.   We have
 \[
 \star \alpha = \alpha_1dx_2\wedge (dx_3-x_1dx_2) - \alpha_ 2 dx_1\wedge (dx_3-x_1dx_2) + \alpha_3e^udx_1\wedge dx_2,
 \]
 and so, since $|\Omega|_{g_u} = e^{-u}$, the equation $d\star(|\Omega|_{g_u} \alpha) =0$ implies
 \begin{equation}\label{eq: coClosedDef}
 \frac{\del}{\del x_1}(e^{-u}\alpha_1) +\left( \frac{ \del}{\del x_2}+x_1\frac{\del }{\del x_3}\right)(e^{-u}\alpha_2) + \frac{\del }{\del x_3}\alpha_3 =0.
 \end{equation}

Next we examine the equation $d\alpha+T\alpha=0$.  Using the formula for $T\alpha$ in~\eqref{eq: TaFinal} we compute 
 \[
 \begin{aligned}
 d\alpha &= \left(\frac{\del \alpha_2}{\del x_1}  - \left(\frac{\del \alpha_1}{\del x_2} + x_1 \frac{\del \alpha_1}{\del x_3}\right)-\alpha_3\right)dx_1\wedge dx_2\\
 &\quad + \left(\frac{\del \alpha_3}{\del x_1} - \frac{\del \alpha_1}{\del x_3}\right) dx_1\wedge (dx_3-x_1dx_2)\\
 &\quad + \left( \left(\frac{\del \alpha_3}{\del x_2} + x_1 \frac{\del \alpha_3}{\del x_3}\right) - \frac{\del \alpha_2}{\del x_3}\right) dx_2\wedge (dx_3-x_1dx_2).
 \end{aligned}
 \]
  Thus, the equation $d\alpha +T\alpha =0$ becomes the following system of differential equations
  \begin{equation}\label{eq: da+Ta}
  \begin{aligned}
   \left(\frac{\del \alpha_2}{\del x_1}  - \left(\frac{\del \alpha_1}{\del x_2} + x_1 \frac{\del \alpha_1}{\del x_3}\right)-\alpha_3\right)&=-\left(\alpha_1\left(\frac{\del u}{\del x_2}+x_1\frac{\del u}{\del x_3}\right)-\alpha_2\frac{\del u}{\del x_1} + \alpha_3 \right) \\
    \left(\frac{\del \alpha_3}{\del x_1} - \frac{\del \alpha_1}{\del x_3}\right)&=-e^{-u}\alpha_2\\
 \left(\frac{\del \alpha_3}{\del x_2} + x_1 \frac{\del \alpha_3}{\del x_3}\right) - \frac{\del \alpha_2}{\del x_3} &=-e^{-u}\alpha_1. 
    \end{aligned}
    \end{equation}
   
    We can rewrite the first equation of~\eqref{eq: da+Ta} as 
    \begin{equation}\label{eq: 12eq}
    \begin{aligned}
    \frac{\del}{\del x_1} (e^{-u}\alpha_2) = \left(\frac{\del }{\del x_2}+x_1\frac{\del }{\del x_3}\right)(e^{-u}\alpha_1)
    \end{aligned}
    \end{equation}
   Let us  denote
    \[
    \del_1= \frac{\del}{\del x_1} \quad \del_2 =  \left(\frac{\del }{\del x_2}+x_1\frac{\del }{\del x_3}\right) \quad \del_3= \frac{\del}{\del x_3}
    \]
    Note that $ [\del_1, \del_2] = \del_3$ while all other Lie brackets are zero.  The trick is to decouple one term from the above system.  We differentiate~\eqref{eq: coClosedDef} in the $\del_1$ direction, and use the commutation relation to obtain
    \[
    \del_1^2(e^{-u}\alpha_1) + \del_2\del_1(e^{-u}\alpha_2) + \del_3(e^{-u}\alpha_2) + \del_3\del_1(\alpha_3)=0
    \]
    Applying \eqref{eq: 12eq} we obtain
    \[
    \del_1^2(e^{-u}\alpha_1) + \del_2\del_2(e^{-u}\alpha_1) + \del_3(e^{-u}\alpha_2) + \del_3\del_1(\alpha_3)=0 . 
    \]
    Finally, by the second equation in~\eqref{eq: da+Ta} we get
    \[
    \del_3\del_1(\alpha_3)= \del_3^2\alpha_1 -\del_3(e^{-u}\alpha_2).
    \]
    Thus, since $u=u(x_1,x_2)$ is independent of $x_3$ we obtain
    \[
    \del_1^2(e^{-u}\alpha_1) + \del_2^2(e^{-u}\alpha_1)  + e^u\del_3^2(e^{-u}\alpha_1)=0. 
    \]
    Thus, by the maximum principle we conclude that $e^{-u}\alpha_1$ is a constant.  On the other hand, a similar calculation, differentiating~\eqref{eq: coClosedDef} in the $\del_2$ direction, yields
    \[
         \del_1^2(e^{-u}\alpha_2) + \del_2^2(e^{-u}\alpha_2) + e^{u}\del_3^2(e^{-u}\alpha_2) = 2 \del_3(e^{-u}\alpha_1)=0
       \]
       since $e^{-u}\alpha_1$ is a constant.  Another application of the maximum principle shows that $e^{-u}\alpha_2$ is a constant.  Evaluating the second and third equations of~\eqref{eq: da+Ta} at the maximum of $\alpha_3$ shows that $e^{-u}\alpha_1= e^{-u}\alpha_2=0$.  Equation~\eqref{eq: coClosedDef} shows that $\del_3\alpha_3=0$, while the second and third equations of~\eqref{eq: da+Ta} show that $\del_1\alpha_3=0=\del_2\alpha_3$, and so $\alpha_3$ is a constant.  On the other hand, since $\alpha_1=\alpha_2=0,$ and $\alpha_3=const.$ is clearly a solution of the system  ~\eqref{eq: coClosedDef} and ~\eqref{eq: da+Ta} we see that the deformation space is indeed $1$-dimensional.  Since the special Lagrangians $L_t$ clearly integrate this deformation, we conclude that the deformation theory is unobstructed in this case.
       \end{proof}

      The proof of Lemma~\ref{lem: Iwasawa} is essentially identical, but simpler, so we will only sketch the proof.
      \begin{proof}[Proof of Lemma~\ref{lem: Iwasawa}]
Computing as above we find that the infinitesimal deformation equations are given by
 \begin{equation}\label{eq: IwaDef}
  \begin{aligned}
  & \left(\frac{\del \alpha_2}{\del x_1}  - \left(\frac{\del \alpha_1}{\del x_2} + x_1 \frac{\del \alpha_1}{\del x_3}\right)-\alpha_3\right)=-\left(\alpha_1\left(\frac{\del u}{\del x_2}+x_1\frac{\del u}{\del x_3}\right)-\alpha_2\frac{\del u}{\del x_1} + \alpha_3 \right) \\
    &\left(\frac{\del \alpha_3}{\del x_1} - \frac{\del \alpha_1}{\del x_3}\right)=-e^{-u}\alpha_2\\
& \left(\frac{\del \alpha_3}{\del x_2} + x_1 \frac{\del \alpha_3}{\del x_3}\right) - \frac{\del \alpha_2}{\del x_3} =e^{-u}\alpha_1,\\
 & \frac{\del}{\del x_1}(e^{-u}\alpha_1) +\left( \frac{ \del}{\del x_2}+x_1\frac{\del }{\del x_3}\right)(e^{-u}\alpha_2) + \frac{\del }{\del x_3}\alpha_3 =0
    \end{aligned}
    \end{equation}      
Note that the only difference between~\eqref{eq: IwaDef} and the deformation equations for $X$ considered above is a sign change on the right hand side of the third equation; cf~\eqref{eq: da+Ta}.  Now one may proceed as in the proof of Lemma~\ref{lem: Fu-Yau} except that the equations for $e^{-u}\alpha_1$ and $e^{-u}\alpha_2$ both decouple (making the case of the Iwasawa manifold comparatively simpler).
\end{proof}

\section{Background on Conifold Transitions}\label{sec: backConifold}
\subsection{Local Setup}
We start by establishing notation for the local geometry of a conifold transition. Let
\[
V_t = \left\{ \sum_{i=1}^4 z_i^2 = t \right\} \subset \mathbb{C}^4.
\]
Local coordinates near $z_4 \neq 0 $ are given by $(z_1,z_2,z_3)$. We will denote
\[
\| z \|^2 = |z_1|^2+|z_2|^2+|z_3|^2+|z_4|^2.
\]
The Candelas-de la Ossa metrics \cite{CandelasdlOssa} on $V_t$ is a family of K\"ahler Ricci-flat metrics that we now review. For $t=1$ the metric is given by the ansatz
\[
\omega_{co,1} = \ddb \psi(\| z \|^2)
\]
where $\psi(s)$ solves an appropriate ODE such that ${\rm Ric}(\omega_{co,1}) = 0$. Once $\omega_{co,1}$ is known, the metrics for all other $t \neq 0$ are given by
\be \label{eq:gcoscaling}
g_{co,t} = |t|^{2/3} S^* g_{co,1}
\ee
with $S: V_t \rightarrow V_1$ the scaling map $S(z) = t^{-1/2} z$. The choice of square root does not affect the discussion. For $t=0$, the ODE for a K\"ahler Ricci-flat metric on $V_0$ has the simple solution $\omega_{co,0} = \ddb r^2$, which is a cone metric of the form $g_{co,0} = dr^2 + r^2 g_L$, with
\[
r(z) = \| z \|^{2/3}.
\]
The metrics $r^{-2} g_{co,t}$ have uniform geometry in the following sense: around each point $\hat{z} \in V_t$, there exists a coordinate ball of uniform size such that the matrix representing $g_{co,t}$ in these coordinates satisfies the estimates
\be \label{eq:r-2gcot}
C^{-1} g_{euc} \leq r^{-2} g_{co,t} \leq C g_{euc}, \quad |\partial^k (r^{-2} g_{co,t})| \leq C_k
\ee
where the constants $C,C_k>0$ are uniform in $t$. (See e.g. \cite{Chuan}). 
\smallskip
\par The model holomorphic volume form on $V_t$ is defined by
\[
\Omega_{mod,t} = {d z_1 \wedge dz_2 \wedge dz_3 \over z_4} \quad {\rm on} \ z_4 \neq 0.
\]
It scales as $S^* \Omega_{mod,1} = t^{-1} \Omega_{mod,t}$ and the normalization on $g_{co,t}$ is such that
\[
| \Omega_{mod,t} |_{g_{co,t}} = 1.
\]
Next, we discuss the vanishing cycle of $V_t$, which is denoted by
\[
L_t = \bigg\{ \| z \|^2 = |t| \bigg\} \subset V_t.
\]
Topologically, the set $L_t$ can be identified with $S^3$, and we denote the embedding by $f_t: S^3 \rightarrow V_t$ with
\[
f_t(x_1,x_2,x_3,x_4) = t^{1/2}(x_1,x_2,x_3,x_4)
\]
where $S^3 = \{ \sum_{i=1}^4 x_i^2 = 1 \} \subset \mathbb{R}^4$. If the parameter $t=|t| e^{i \theta}$, then
\[
f_t^* \omega_{co,t} =0, \quad f_t^* ({\rm Im} \, e^{-i \theta} \Omega_{mod,t}) = 0,
\]
and so the 3-cycle $L_t$ is special Lagrangian with respect to $(\omega_{co,t},\Omega_{mod,t})$. Let $h_{t} = f_t^* g_{co,t}$ be the induced metric on $S^3$. By (\ref{eq:gcoscaling}), there is the simple scaling relation
\[
  h_{t}  = |t|^{2/3} h_{1}.
\]

\subsection{Global Setup}
We now consider a global conifold transition
\be
\label{top-change}
X \rightarrow \underline{X} \rightsquigarrow X_t,
\ee
as described in \cite{Friedman}. We start with $X$, a simply connected compact K\"ahler Calabi-Yau manifold of complex dimension 3. At the level of topology, a conifold transition is a surgery which replaces embedded 2-spheres with 3-spheres. At the level of complex geometry, we start by contracting $k$ disjoint $(-1,-1)$ curves $C_i \subset X$ to points $p_i$. This produces a singular space $\underline{X}$ which has singularities at $p_i$, and a neighborhood of each $p_i$ can be identified with a neighborhood of
\[
  V_0 = \{ \sum_{i=1}^4 z_i^2 =0 \} \subset \mathbb{C}^4.
\]
Then the compact complex manifolds $X_t$ are the smooth fibers of a holomorphic family $\mathcal{X} \rightarrow \Delta=\{ t \in \mathbb{C} : |t| <1 \}$ with central fiber $\underline{X}$.  Near the singular points the smoothing is locally biholomorphic to the model smoothing $V_t= \{ \sum z_i^2 =t \}$ \cite{KSchles}. It was proved by Friedman \cite{Friedman} and Tian \cite{Tian92} that the existence of this smoothing $X_t$ is implied by the existence of constants $\lambda_i \neq 0$ such that 
\be \label{Friedmans-cond}
\sum_{i=1}^k \lambda_i [C_i] = 0 
\ee
in $H_2(X,\mathbb{R})$. Note that $X_t$ may be non-K\"ahler in general.
\begin{ex}
A simple example of a conifold transition is given in \cite{CGH,GMS}. Consider the singular quintic $\underline{X} \subset \mathbb{P}^4$ given by
\be \label{eq:singquintic}
x_3 g(x) + x_4 h(x) = 0
\ee
for generic homogeneous degree 4 polynomials $g$, $h$. The singularities of $\underline{X}$ are locally modelled by $V_0=\{ \sum_i z_i^2 = 0 \} \subset \mathbb{C}^4$. There are two ways to resolve these singularities. One way is to perform a small resolution by blow-up which produces a smooth Calabi-Yau $X$. The other is to deform $\underline{X}$ to a smooth quintic $X_t$ by introducing e.g. $t \sum x_i^5$ on the right-hand side of (\ref{eq:singquintic}). This is an example of a conifold transition $X \rightarrow \underline{X} \rightsquigarrow X_t$ where both sides $X,X_t$ are K\"ahler.
\end{ex}

\begin{ex}
An example of a transition from K\"ahler to non-K\"ahler is given in \cite{Friedman}. For this example, we take the initial threefold to be a smooth quintic $X \subset \mathbb{P}^4$. The second Betti number is ${\rm rk} H_2(X,\mathbb{R})=1$, and by Friedman’s criteria (\ref{Friedmans-cond}) a choice of $k \geq 2$ disjoint $(-1,-1)$ curves leads to a contraction and smoothing $X \rightarrow \underline{X} \rightsquigarrow X_t$. Friedman showed that $H_2(X_t,\mathbb{Z})=\mathbb{Z} / d \mathbb{Z}$ for an integer $d$ depending on the choice of curves, and so $X_t$ cannot admit a K\"ahler metric.
\end{ex}

To study the geometry of $X \rightarrow \underline{X} \rightsquigarrow X_t$, Fu-Li-Yau \cite{FLY} proposed to use balanced metrics and proved their existence through conifold transitions. In the case when $X_t$ admits a K\"ahler metric, the geometry of a transition with respect to K\"ahler Ricci-flat metrics was studied in \cite{RongZhang,Song}. Fu-Li-Yau's uniform approach works for all transitions, and this only requires the condition (\ref{Friedmans-cond}) on the initial configuration of the holomorphic curves; there is no requirement to distinguish whether $X_t$ can admit a K\"ahler structure. 

\subsubsection{Fu-Li-Yau metric} Let $r: \mathcal{X} \rightarrow (0,\infty)$ be a function which agrees with $r = \| z \|^{2/3}$ near the singular points $p_i$ and $r^{-1}(0) = \{ p_1,\dots,p_k \}$. We assume that the set $X_t \cap \{ r < 1 \}$ contains $k$ components which we each biholomorphically identify with $V_t \cap \{ r < 1 \}$. Each $i$th component of $X_t \cap \{ r < 1 \}$ contains a vanishing cycle $L_{i,t}=\{ \| z \|^2 = |t| \}$.

\par Let $\omega_t$ denote the Fu-Li-Yau balanced metric on $X_t$ satisfying $d \omega_t^2 = 0$ and let $g_t$ denote the associated Riemannian metric. We now recall the properties of $g_t$ that we will need. First, on compact sets $K \subset \mathcal{X}$ disjoint from the singularities, the geometry $(X_t,g_t)$ is uniformly bounded in $t \in \Delta$ (but dependent on $K$) in all $C^k$ norms with respect to a reference metric on $\mathcal{X}$. Next, near the singularities, the metrics $g_t$ are locally modelled on the Candelas-de la Ossa metrics $g_{co,t}$. To be precise, given a singular point $p_i$, there is $\lambda >0$ such that the following estimate (stated as Lemma 6.6 in \cite{CPY}, which can be derived from the estimates in \cite{FLY} and \cite{Chuan}) holds on the component of $\mathcal{X} \cap \{ r < 1 \}$ containing $p_i$:
\be \label{eq:cot2check}
| \nabla_{g_{co,t}}^k (g_t - \lambda g_{co,t})|_{g_{co,t}} \leq C_k |t|^{2/3} r^{-k}.
\ee
Here $C_k>0$ is independent of $t$, and $g_{co,t}$ is the Ricci-flat metric constructed by Candelas-de la Ossa. Finally, we recall that $g_t$ converges uniformly on compact sets as $t \rightarrow 0$ to a metric $g_0$ on $\underline{X}_{reg}$ which satisfies $d \omega_0^2=0$ and agrees with a multiple of $g_{co,0}$ in a neighborhood of each singularity.

\subsubsection{Holomorphic volume form} The complex manifolds $X_t$ emerging from a conifold transition have trivial canonical bundle \cite{Friedman}. We will need a local description of the holomorphic volume forms $\Omega_t$ on $X_t$, and we give here an explicit construction of the trivializing section of $K_{X_t}$. Let $\Omega_0$ be the holomorphic volume form on $\underline{X}$ obtained by restricting the holomorphic volume form on the K\"ahler Calabi-Yau $X$. Near a singular point $p$ of $\underline{X}$, we identify a neighborhood of $p$ with $V_0 \cap \{ r \leq 1 \}$, and we can write
\[
\Omega_0 = h(z) \Omega_{mod,0}
\]
for a non-vanishing holomorphic function $h(z)$ defined on $V_0 \subset \mathbb{C}^4 \cap \{ r \leq 1 \}$. The singularity at the origin comes from contracting a holomorphic curve $C_i \subset X$, hence $h(z)$ assumes a value $h(0) \in \mathbb{C}^*$. Furthermore, since $V_0$ is a normal variety, $h(z)$ can be defined in an open set $U \subset \mathbb{C}^4$ containing the origin; see, for example \cite[Chapter II]{Demailly}.

\par We now smoothly extend $\Omega_0$ to $\mathcal{X}$. We assume that there is only one singularity $p \in \underline{X}$ for simplicity of notation. Let $\delta>0$ be such that $h$ defines a holomorphic function on $\{ r < 2 \delta \} \subset \mathbb{C}^4$, and let $\zeta$ denote a cutoff function on $\mathbb{C}^4$ such that $\zeta \equiv 1$ on $\{ r \leq \delta \}$ and $\zeta \equiv 0$ on $\{ r > 2 \delta \}$.  The family $X_t \backslash \{ r \leq \delta \}$ is smooth, and we let
\[
\Phi_t: X_0 \backslash \{r \leq \delta \} \rightarrow X_t \backslash \{r \leq \delta \}
\]
be a smoothly varying family of diffeomorphisms with $\Phi_0 = id$. We introduce the following 3-form on $X_t$:
\[
\Psi_t = \zeta h(z) \Omega_{mod,t} + (\Phi_t^{-1})^* \big((1-\zeta)\Omega_0 \big)
\]
and to obtain a smooth non-vanishing section of $K_{X_t}$, we take the $(3,0)$ part. The section $\Psi_t^{3,0}$ is non-vanishing for small $t$, since this holds on $X_t \cap \{ r \leq \delta \}$ because $\Psi_t^{3,0}=h(z) \Omega_{mod,t}$ there, and on $X_t \backslash \{ r \leq \delta \}$ the forms are varying smoothly in $t$ and $\Psi_0^{3,0}= \Omega_0$ is non-vanishing at $t=0$. We let $\sigma_t$ denote the representative of the section $\Psi_t^{3,0} \in \Gamma(X_t,K_{X_t})$ in a local trivialization, i.e.
\[
\Psi_t^{3,0} = \sigma_t \, dz_1 \wedge dz_2 \wedge dz_3.
\]
The next step is to correct the smooth section $\sigma_t$ to a holomorphic section. For this, first we note that $ \sigma_t^{-1} \bar{\partial} \sigma_t$ defines a $\bar{\partial}$-closed $(0,1)$ form on $X_t$. Next, by Lemma 8.2 in \cite{Friedman}, we have $H^1(X_t,\mathcal{O}_{X_t})=0$ since $H^1(X,\mathcal{O}_{X})=0$ because $X$ is simply connected. Since the Dolbeault cohomology vanishes $H_{\bar{\partial}}^{0,1}(X_t)=0$, we can solve the $\bar{\partial}$-equation
\be \label{eq:u-defn}
\bar{\partial} u_t = - {1 \over \sigma_t} \bar{\partial} \sigma_t, \quad \int_{X_t} u_t \, \Psi_t^{3,0} \wedge \overline{\Psi_t^{3,0}} =0 ,
\ee
for a smooth function $u_t \in C^\infty(X_t)$. The corrected section
\[
\Omega_t = e^{u_t} \sigma_t \, dz_1 \wedge dz_2 \wedge dz_3
\]
satisfies $\bar{\partial} \Omega_t = 0$ and equips $X_t$ with a holomorphic volume form.

\subsection{Estimate of the holomorphic volume form}
For later use, we will need to compare the global section $\Omega_t$ on $X_t$ to the local model $\Omega_{mod,t}$ on $V_t$. We will use the convention where $C$ denotes a generic constant which may change line by line. In this section, we will use weighted H\"older norms on $X_t$; for a function $u: X_t \rightarrow \mathbb{R}$, we denote:
\bea
& \ & \| u \|_{C^\gamma_{-\beta}} = |r^\beta u|_{L^\infty(X_t)} + [u]_{C^\gamma_{-\beta-\gamma}} \nonumber\\
& \ & [u]_{C^\gamma_{-\beta-\gamma}} = \sup_{x \neq y} \bigg[\min(r(x),r(y))^{\beta+\gamma} {|u(x)-u(y)| \over d_{g_t}(x,y)^\gamma} \bigg] \nonumber\\
& \ & \| u \|_{C^{2,\gamma}_{-\beta}} = |r^\beta u|_{L^\infty(X_t)} + |r^{\beta+1} \nabla u|_{L^\infty(X_t,g_t)} + |r^{\beta+2} \nabla^2 u|_{L^\infty(X_t,g_t)} + [\nabla^2 u]_{C^\gamma_{-\beta-2-\gamma}} \nonumber\\
& \ & [\nabla^2 u]_{C^\gamma_{-\beta-2-\gamma}} = \sup_{x \neq y} \bigg[ \min(r(x),r(y))^{\beta+2+\gamma} {|\nabla^2 u(x)-\nabla^2 u(y)| \over d_{g_t}(x,y)^\gamma} \bigg].\nonumber
\eea
When writing $|\nabla^2 u(x)- \nabla^2 u(y)|$, we use parallel transport along geodesics with respect to $g_t$ to compare the values $\nabla^2 u$ at different points. Covariant derivatives are with respect to $g_t$.
\smallskip
\par We now state the main estimate for comparing $\Omega_t$ to the local model $\Omega_{mod,t}$.

\begin{lem} \label{lem:Omega-est} There exists $\tau \in \mathbb{C}^*$ and $C_k>1$ near each singular point such that
  \be \label{eq:Omegamod2Omega}
\sup_{\{\| z\|^2= t\}} | \nabla^k_{g_{co,t}} ( \Omega_t - \tau \Omega_{mod,t} ) |_{g_{co,t}} \leq C_k |t|^{1/2} |t|^{-k/3},
\ee
for all $t$ small enough.
\end{lem}
\begin{proof} Let $p$ be a singular point of $\mathcal{X}$. In the neighborhood $\{ r < \delta \}$ containing $p$, the holomorphic volume form appears on $V_t$ as
  \[
\Omega_t = e^{u_t} h(z) \Omega_{mod,t}.
  \]
We will take $\tau = h(0)$. Thus
  \[
|\nabla^k_{g_{co,t}} (\Omega_t - \tau \Omega_{mod,t})|_{g_{co,t}} = |\nabla^k_{g_{co,t}}  (e^{u_t} h - h(0))|_{g_{co,t}}
  \]
where we used that $\Omega_{mod,t}$ is parallel with respect to $g_{co,t}$ and has constant norm thanks to the Calabi-Yau condition.  We will estimate $e^{u_t}h(z)-h(0)$. We start with $k=0$.
  \be \label{eq:Omega-est0}
|e^{u_t} h(z) - h(0)| \leq e^{u_t} |h(z)-h(0)| + |h(0)| |e^{u_t} - 1|.
  \ee
  We claim that there exists $C>1$ independent of $t$ such that
  \be \label{eq:smallu}
\sup_{X_t} |u_t| \leq C |t|^{2/3}.
\ee
This follows from the following weighted estimate: for any $\beta\in (0,2)$ and $\gamma \in (0,1)$, there exists $C>1$ uniform in $t$ such that
\be \label{eq:invertLaplace}
\| u \|_{C^{2,\gamma}_{-\beta}} \leq C \| \Delta_{g_t} u \|_{C^\gamma_{-2-\beta}}, 
\ee
for all  $u \in C^{2,\gamma}(X_t)$ satisfying
\[
\int_{X_t} u \, \Psi_t^{3,0} \wedge \overline{\Psi_t^{3,0}} = 0.
\]
 We are using the notation $\Delta_g u := g^{i \bar{j}} \partial_i \partial_{\bar{j}} u$ for the complex Laplacian, which is not the Laplacian of the Levi-Civita connection for non-K\"ahler $g$. The proof of estimate (\ref{eq:invertLaplace}) is a standard blow-up argument, and we will provide the proof in Lemma \ref{lem:invertLaplace} below. We now assume (\ref{eq:invertLaplace}) and continue the proof of (\ref{eq:smallu}) and Lemma \ref{lem:Omega-est}.
\smallskip
\par Our correction function $u_t$ defined in (\ref{eq:u-defn}) satisfies
\[
\Delta_{g_t} u_t = - g_t^{j \bar{k}} \partial_j (\sigma_t^{-1} \partial_{\bar{k}} \sigma_t) := \psi_t.
\]
By construction of $\sigma_t$, this right-hand side $\psi_t$ vanishes identically in a neighborhood of the singular points. On $X_t \backslash \{ r < \delta \}$, the geometry $g_t$ is bounded and $\bar{\partial} \sigma_t$ is zero everywhere at $t=0$. Therefore
\[
\| \psi_t \|_{C^\gamma_{-2-\beta}} \leq C |t|.
\]
The weighted estimate (\ref{eq:invertLaplace}) implies $r^{\beta} |u_t| \leq C |t|$, which proves (\ref{eq:smallu}) (taking $\beta=1$ to be concrete) since $\| z \|^2 \geq |t|$.
\smallskip
\par Using estimate (\ref{eq:smallu}) and $|h(z)-h(0)| \leq C \|z\|$, we conclude that for all $t$ small enough we can estimate
  \[
|e^{u_t}h(z)  - h(0) | \leq C \|z\| + C |t|^{2/3}
  \]
  on $\{ r < \delta \}$. It follows that for $t$ small enough, (\ref{eq:Omega-est0}) becomes
\[
\sup_{ \{ \| z \|^2 \leq 2 |t| \}} |\Omega_t - h(0) \Omega_{mod,t}|_{g_{co,t}}  \leq C |t|^{1/2}.
\]
We take $\tau = h(0)$ to obtain the estimate in the lemma with $k=0$.
  \smallskip
\par For higher order estimates, fix a point $\hat{z} \in V_t$ such that $\| \hat{z} \|^2 = |t|$. Suppose $z_4 \neq
0$. Let $\hat{x} \in V_1$ be given by $\hat{x}=S(\hat{z})$ where $S(z)=t^{-1/2}z$. Since $\hat{x}_4 \neq 0$, we may take coordinates
$(x_1,x_2,x_3)$ on $V_1 = \{ \sum_{i=1}^4 x_i^2 =1\}$. Denote $\hat{p} = (\hat{x}_1,\hat{x}_2,\hat{x}_3) \in \mathbb{C}^3$.
\smallskip
\par Let $h_t$ be the holomorphic function defined on $B_1(\hat{p}) \subset \mathbb{C}^3$ given by $h_t(x) = (e^{u_t} h)(t^{1/2}x)$ (recall $\bar{\partial} u_t = 0$ near the singularities).  The local estimates for holomorphic functions
\[
\sup_{B_{1/2}(\hat{p})}  |h_t-h(0)|_{C^k(g_{euc})} \leq C_k  \sup_{B_1(\hat{p})}| h_t - h(0) |_{L^\infty} \leq C_k |t|^{1/2},
\]
can be converted on the compact geometry $(V_1,g_{co,1}) \cap \{ \| x \| = 1 \}$ to the estimate
\[
|\nabla^k_{g_{co,1}} h_t|_{g_{co,1}}(\hat{p}) \leq C_k  |t|^{1/2}, \quad k \geq 1.
\]
Pulling back this estimate to $V_t$ via $S^*$ and using (\ref{eq:gcoscaling}) gives
\[
\sup_{\{ \| z \|^2 = |t| \}}  |\nabla^k_{g_{co,t}} e^{u_t} h |_{g_{co,t}} \leq C_k  |t|^{1/2} |t|^{-k/3}, \quad k \geq 1,
\]
which implies the higher order estimates stated in the lemma. 
\end{proof}

It remains to prove the uniform estimate for inverting the Laplacian.

\begin{lem} \label{lem:invertLaplace}
For any $\beta\in (0,2)$ and $\gamma \in (0,1)$, there exists $C>1$ uniform in $t$ such that
\[
\| u \|_{C^{2,\gamma}_{-\beta}} \leq C \| \Delta_{g_t} u \|_{C^{0,\gamma}_{-2-\beta}}, 
\]
for all  $u \in C^{2,\gamma}(X_t)$ satisfying
\[
\int_{X_t} u \, \Psi_t^{3,0} \wedge \overline{\Psi_t^{3,0}} = 0.
\]
  \end{lem}

  \begin{proof} This type of blow-up argument is now standard, e.g. \cite{Spotti, YLi, CPY}. We give the proof for the sake of completeness. The first step is to show the uniform in $t$ Schauder estimate
 \[
\| u \|_{C^{2,\gamma}_{-\beta}} \leq C ( \| r^\beta u \|_{L^\infty} + \| \Delta_{g_t} u \|_{C^{0,\gamma}_{-2-\beta}}), \quad u \in C^{2,\gamma}(X_t).
\]
This estimate holds on $X_t \cap \{ r \geq 1 \}$ by the standard Schauder estimates since the geometry of $g_t$ is uniformly bounded there for $0 \leq |t| \leq 1$. On $X_t \cap \{ r < 1 \}$, which is identified with a subset of $V_t= \{ \sum_{i=1}^4 z_i^2=t \}$, the metric $r^{-2} g_{co,t}$ has bounded geometry (\ref{eq:r-2gcot}). The Schauder estimates in local coordinates then give the weighted estimate with respect to the metric $g_{co,t}$. Then estimate (\ref{eq:cot2check}) can be used to go from norms with respect to $g_{co,t}$ to $g_t$ once $t$ is small enough. 
\smallskip
\par To prove the estimate without the $\| r^\beta u \|$ term, suppose by contradiction that there exists a sequence $u_i \in C^{2,\gamma}(X_{t_i})$ with $t_i \rightarrow 0$ and
\[
1 \geq M_i \| \Delta_{g_{t_i}} u_i \|_{C^{0,\gamma}_{-2-\beta}}, \quad \| u_i \|_{C^{2,\gamma}_{-\beta}}=1, \quad \int_{X_{t_i}} u_i \, \Psi_{t_i}^{3,0} \wedge \overline{\Psi_{t_i}^{3,0}}=0
\]
for constants $M_i \rightarrow \infty$.  Let $z_i \in X_{t_i}$ be a point where $r^\beta u$ attains its supremum on $X_{t_i}$. The uniform Schauder estimate implies
\[
{1 \over C} \leq r(z_i)^\beta |u_i(z_i)|.
\]
We start with the case when $\liminf r(z_i) >0$. After taking a subsequence, we have $z_i \rightarrow z_0 \in \underline{X}$ with $r(z_0)>0$. By the construction in \cite{FLY}, the metrics $g_{t_i}$ converge as $t_i \rightarrow 0$ to a metric $g_0$ on $\underline{X}$ which is balanced and agrees with $g_{co,0}$ near the singularities, and the convergence is uniform on compact sets disjoint from the singularities (see also \S 2, \S 6.1  in \cite{CPY} for further details). Also, we note that $\Psi_t^{3,0}$ converges back to $\Omega_0$ as $t \rightarrow 0$ locally uniformly. 
\smallskip
\par After taking a limit on compact sets disjoint from the nodes, we obtain a function $u_0 \in C^{2,\gamma}_{-\beta}(\underline{X})$ on the singular space satisfying
\be \label{eq:integral-limit}
\Delta_{g_0} u_0 = 0, \quad \int_{\underline{X}} u_0 \, \Omega_0 \wedge \overline{\Omega_0} =0. 
\ee
Here is justification for the integral. For fixed $0<\delta \ll 1$, uniform convergence implies
\bea
\int_{\{ r \geq \delta \} \cap \underline{X}} u_0 \, \Omega_0 \wedge \overline{\Omega_0} &=& \lim_{t \rightarrow 0} \int_{\{ r \geq \delta \} \cap X_t} u_t \, \Psi_t^{3,0} \wedge \overline{\Psi_t^{3,0}} \nonumber\\
&=& - \lim_{t \rightarrow 0} \int_{\{ r < \delta \} \cap X_t} u_t \, \Psi_t^{3,0} \wedge \overline{\Psi_t^{3,0}}. \nonumber
\eea
Since $\Psi_t^{3,0}=h(z) \Omega_{mod,t}$ on $\{ r < {1 \over 2} \}$ and $\Omega_{mod,t} \wedge \overline{\Omega_{mod,t}} = d {\rm vol}_{g_{co,t}}$, we have
\[
\int_{\{ r \geq \delta \} \cap \underline{X}} u_0 \, \Omega_0 \wedge \overline{\Omega_0} = O(\delta^{6-\beta})
\]
Sending $\delta \rightarrow 0$ gives (\ref{eq:integral-limit}) since $\beta < 6$.
\smallskip
\par We now obtain a contradiction. Let $\zeta_\epsilon$ be a cutoff function such that $\zeta_\epsilon \equiv 0$ on $\{ r < \epsilon \}$ , $\zeta_\epsilon \equiv 1$ on $\{ r \geq 2 \epsilon \}$, and $|\nabla \zeta_\epsilon|_{g_0} \leq C \epsilon^{-1}$. Then
\[
\int_{\underline{X}} \zeta |\nabla u_0|^2_{g_0} d {\rm vol}_{g_0}\leq \int_{\{\epsilon< r < 2 \epsilon \} } u_0 |\nabla u_0| |\nabla \zeta| d {\rm vol}_{g_{co,0}}\leq C  \epsilon^{-\beta} \epsilon^{-1-\beta} \epsilon^{-1} \epsilon^6,
\]
using the balanced condition on $g_0$ to integrate by parts, the decay of $u_0$, and that $g_0$ agrees with the cone metric $g_{co,0}$ near the singularities. Since $\beta \in (0,2)$, this shows that $\bar{\partial} u_0 \equiv 0$. By Hartog's theorem, $u_0$ defines a holomorphic function on the small resolution $X \rightarrow \underline{X}$. Thus $u_0$ is a constant function, and hence $u_0 \equiv 0$ by (\ref{eq:integral-limit}). But $|u_0(z_0)| \geq C^{-1} r(z_0)^{-\beta}$ which is a contradiction.
\smallskip
\par We must therefore rule out the possibility that $\liminf r(z_i) = 0$. We may assume that all $z_i$ lie in $V_{t_i} \cap \{ \| z \|^2 \leq 1 \}$. On this set, we use (\ref{eq:cot2check}) to convert estimates with respect to the global metrics $g_t$ into estimates with respect to the local model metrics $g_{co,t}$. For example, the sequence $\{ u_i \}$ restricted to $V_{t_i} \cap \{ \| z \|^2 \leq 1 \}$ satisfies
\[
\| u_i \|_{C^{2,\gamma}_{-\beta}(g_{co,t_i})} \leq 1.
\]
We can also estimate the Laplacian with respect to the model metric $g_{co,t}$ by using (\ref{eq:cot2check}):
\bea
| r^{2+\beta} \Delta_{g_{co,t_i}} u_i| &\leq& | r^{2+\beta} \Delta_{g_{t_i}} u_i| + |g_{t_i}^{-1} - g_{co,t_i}^{-1}|_{g_{co,t_i}} |r^{2+\beta} \nabla^2 u_i|_{g_{co,t_i}} \nonumber\\
&\leq& M_i^{-1} + C |t_i|^{2/3}.
\eea
Let $r(z_i)=\lambda_i$ and define the scaling map $S_{\lambda_i}: V_{t_i \lambda_i^{-3}} \rightarrow V_{t_i}$ by $S(x)= \lambda_i^{3/2} x$. Consider
\[
\tilde{u}_i : V_{t_i \lambda_i^{-3}} \cap \{ \| x \|^2 \leq \lambda_i^{-3} \} \rightarrow \mathbb{R}
\]
given by
\[
\tilde{u}_i(x) = \lambda_i^\beta u_i (\lambda_i^{3/2} x).
\]
This sequence of functions satisfies
\[
  \| \tilde{u}_i \|_{C^{2,\alpha}_{-\beta}(g_{co,t_i \lambda_i^{-3}})} \leq 1, \quad | r^{2+\beta} \Delta_{g_{co,t_i \lambda_i^{-3}}} \tilde{u}_i | \rightarrow 0, \quad |\tilde{u}_i(x_i)| \geq C^{-1}
\]
for $x_i = \lambda_i^{-3/2} z_i$ satisfying $\| x_i \|=1$. This can be shown using the $g_{co,t}$ rescaling relation (\ref{eq:gcoscaling}) which here takes the form $g_{co,t \lambda^{-3}} = \lambda^{-2} S_{\lambda}^* g_{co,t}$.
\smallskip
\par $\bullet$ Suppose $t_i \lambda_i^{-3} \rightarrow \kappa >0$  and $x_i \rightarrow x_\infty \in V_\kappa$ after taking a subsequence. Then from $\{ \tilde{u}_i \}$ we obtain a limiting function $u_\infty$ on $V_\kappa$ which satisfies $\Delta_{g_{co,\kappa}} u_\infty =0$. Since $|u_\infty| \leq C r^{-\beta}$ for $\beta>0$, then $u_\infty \equiv 0$, which is a contradiction since $|u_\infty(x_\infty)| \geq C^{-1}$.
\smallskip
\par $\bullet$ Suppose $t_i \lambda_i^{-3} \rightarrow 0$ and $x_i \rightarrow x_\infty \in V_0$ (with $\| x_\infty \|=1$) after taking a subsequence. Then from $\{ \tilde{u}_i \}$ we obtain a limiting function $u_\infty$ on $V_0$ which satisfies $\Delta_{g_{co,0}} u_\infty=0$. The limiting metric $g_{co,0}$ is a cone metric \cite{CandelasdlOssa} of the form $g_{co,0} = dr^2 + r^2 g_L$. The Riemannian cone $(V_0,g_{co,0})$ of dimension $n=6$ does not admit non-zero harmonic functions with singularity rate $|u| \leq C r^{-\beta}$ in the gap $\beta \in (0,n-2)$. Therefore $u_\infty \equiv 0$, which again contradicts $|u_\infty(x_\infty)| \geq C^{-1}$.
    
    \end{proof}

\section{Special Lagrangian Spheres}\label{sec: spheres}
Let $X \rightarrow \underline{X} \rightsquigarrow X_t$ denote a conifold transition. Equip the complex manifold $X_t$ with the Fu-Li-Yau \cite{FLY} balanced metric $\omega_t$ (solving $d \omega_t^2=0$) and holomorphic volume form $\Omega_t$ described above. Choose one of the vanishing cycles $L_t = \{ \| z \|^2 = |t| \}$ and let $f_t: S^3 \rightarrow X_t$ denote its embedding. In this section, we will prove our main theorem.

\begin{thm}
When $t$ is small enough, the vanishing cycle $L_t$ can be deformed to a special 3-cycle $\tilde{L}_t \subset X_t$ solving the supersymmetric equations
\[
  \tilde{\omega}_t|_{\tilde{L}_t} = 0, \ ({\rm Im} \, e^{-i \hat{\theta}} \Omega_t)|_{\tilde{L}_t} = 0, \ d (| \Omega_t |_{\tilde{g}_t} \tilde{\omega}^2_t)=0
\]
for an angle $e^{i \hat{\theta}} \in S^1$ and hermitian metric $\tilde{\omega}$ on $X_t$.
\end{thm}

As noted in \S \ref{section:defn}, the conformally balanced metric $\tilde{\omega}_t$ can be obtained by conformal change of the Fu-Li-Yau balanced metric $\omega_t$, and so for simplicity we will work with $\omega_t$ (and show $\omega_t|_{\tilde{L}}=0$) rather than
$\tilde{\omega}_t$.

\subsection{Setup and notation}
The analysis in this section will entirely take place at a chosen vanishing cycle. We rotate coordinates such that a neighborhood of this vanishing cycle is identified with $V_t$ for $t>0$. We use the corresponding constant $\tau \in \mathbb{C}^*$ from Lemma \ref{lem:Omega-est} to change the normalization of $\Omega_t$ so that we may assume $\tau=1$. Similarly, we rescale $\omega_t$ and assume $\lambda=1$ in (\ref{eq:cot2check}).
\smallskip
\par We first note that $\int_{L_t} \Omega_t \neq 0$ for $t$ small enough. This is because
\[
\int_{L_t} \Omega_t = \int_{L_t} \Omega_{mod,t} + \int_{L_t} O(|t|^{1/2}) {\rm dvol}_{L_t} 
\]
by (\ref{eq:Omegamod2Omega}), and $\Omega_{mod,t}|_{L_t} = {\rm dvol}_{L_t}$. We let $e^{-i\hat{\theta}}$ be the angle of this vanishing cycle determined by
\be \label{eq:angle-cond}
{\rm Im} \bigg( e^{-i \hat{\theta}}\int_{L_t} \Omega_t \bigg) = 0, \quad {\rm Re} \bigg( e^{-i \hat{\theta}}\int_{L_t} \Omega_t \bigg) >0.
\ee
Let $h_{t}$ be the induced metric $f_t^* g_{co,t}$. For a given 1-form on $S^3$ denoted $\alpha$, we associate a vector field $V^i = h_{t}^{ik}\alpha_k$ on $S^3$ and a vector field $\xi^\alpha = (J_t)^\alpha{}_\beta \partial_i f^\beta_t V^i$ on $L_t \subset X_t$. We will sometimes drop the $t$ subscript to ease notation and simply write e.g. $f$, $L$. 
\smallskip
\par Let $\varphi[\alpha]$ be the map from $L$ to $X_t$ given by $\varphi[\alpha](p) = \exp_p(\xi_p)$, where $\exp$ is the Riemannian exponential map of $g_{co,t}$. The composition
\[
\varphi[\alpha] \circ f: S^3 \rightarrow X_t
\]
associates to $\alpha \in \Lambda^1(S^3)$ a deformation of the submanifold $L$. Define
\[
  \F: \Lambda^1(S^3) \rightarrow \Lambda^*(S^3)
\]
by the formula
\begin{equation}\label{eq: opF}
\F(\alpha) = -(\varphi[\alpha] \circ f)^* \omega_t - \star_{h_t} (\varphi[\alpha] \circ f)^*{\rm Im} \, e^{-i \hat{\theta}} \Omega_t.
\end{equation}
To deform $f: S^3 \rightarrow X_t$ into a special Lagrangian 3-cycle, we will look for a solution to $\F(\alpha)=0$ for small $|t| \ll 1$. We start with a lemma which states that $\F(0)$ is close to zero when the parameter $t$ is small. Thus the vanishing cycle $L_t$ is an approximate solution to the special Lagrangian equation.

\begin{lem} \label{lem:approx-soln}
  The approximate solution $f: S^3 \rightarrow X_t$ satisfies
  \[
| f^* \omega_t |_{h_{t}} + |t|^{1/3} |\nabla_{h_t} f^* \omega_t|_{h_t} \leq C |t|^{2/3}.
  \]
  \[
|f^* {\rm Im} \, e^{-i\hat{\theta}} \Omega_t|_{h_{t}} + |t|^{1/3} | \nabla_{h_t} f^* {\rm Im} \, e^{-i\hat{\theta}} \Omega_t|_{h_{t}} \leq C |t|^{1/2}.
  \] 
\end{lem}
\begin{proof}
  \par The hermitian form $\omega_t$ can be estimated using (\ref{eq:cot2check})
 \[
|f^* \omega_t|_{h_{t}} = |f^* (\omega_t-\omega_{co,t})|_{f^* g_{co,t}} \leq C |t|^{2/3},
  \]
  since $f^* \omega_{co,t}=0$. The estimate on $\nabla f^* \omega_t$ is similar.
  \smallskip
  \par Next, we estimate $e^{-i \hat{\theta}} \Omega_t$. From (\ref{eq:angle-cond}), we have
  \[
0 = {\rm Im} \,  \int_{L_t} \bigg( e^{-i \hat{\theta}}  \Omega_{mod,t}+ (e^{-i \hat{\theta}} \Omega_t - e^{-i \hat{\theta}}  \Omega_{mod,t}) \bigg).
\]
Therefore (\ref{eq:Omegamod2Omega}) and $\Omega_{mod,t}|_{L_t} = d {\rm vol}_{L_t}$ give the estimate
\be \label{eq:zetaim-est}
\bigg| {\rm Im} \, e^{-i \hat{\theta}}  \bigg| \leq C|t|^{1/2}.
\ee
We can therefore estimate
\bea
|f^* {\rm Im} \, e^{-i \hat{\theta}} \Omega_t|_{h_{t}} &\leq& |f^* {\rm Im} \, (e^{-i \hat{\theta}} \Omega_t-e^{-i \hat{\theta}}  \Omega_{mod,t})|_{f^* g_{co,t}} + |f^* {\rm Im} \, e^{-i \hat{\theta}}  \Omega_{mod,t} |_{f^*g_{co,t}} \nonumber\\
&\leq& |\Omega_t -  \Omega_{mod,t}|_{g_{co,t}} + |{\rm Im} \, e^{-i \hat{\theta}} | \leq C |t|^{1/2},
\eea
using (\ref{eq:Omegamod2Omega}). Higher order estimates involving derivatives are similar.
\end{proof}

An analogous argument to the one used in the proof of the previous lemma (using (\ref{eq:angle-cond})) implies ${\rm Re} \, e^{-i \hat{\theta}} \geq - C|t|^{1/2}$. We combine this with (\ref{eq:zetaim-est}) and record the following estimate for future use:
\be \label{eq:zeta-est}
|1-{\rm Re} \, e^{-i\hat{\theta}}| \leq C|t|^{1/2}, \quad  |{\rm Im} \, e^{-i\hat{\theta}}| \leq C |t|^{1/2}.
\ee
Next, we will study and estimate the derivative of $\F$. Before doing this, we note that the range of $\F$ is contained in the subspace of differential forms on $(S^3,h_t)$ generated by $d$ and $d^\dagger_{h_t}$.
\begin{lem} \label{lem:rangeF}
For the map $\F$ defined in~\eqref{eq: opF} we have
  \[
    \F: \Lambda^1(S^3) \rightarrow {\rm Im} (d+d^\dagger_{h_{t}}).
  \]
\end{lem}
\begin{proof}
By Hodge theory, we need to show that we can write $\F(\alpha) = d \chi_1 + d^\dagger \chi_2$ for $\chi_1,\chi_2 \in \Lambda^*(S^3)$. First, we do have
\[
  (\varphi[\alpha] \circ f)^* \omega_t \in d \Lambda^1 \oplus d^\dagger \Lambda^3,
\]
since $H^2(S^3,\mathbb{R})=0$. Next, we verify that
\[
 \star (\varphi[\alpha] \circ f)^* {\rm Im} \, e^{-i \hat{\theta}} \Omega_t \in d^\dagger \Lambda^1.
\]
For this, we need to show it is orthogonal to the constant functions. The identity
\[
\int_{S^3} (\varphi[\alpha] \circ f)^* {\rm Im} \, e^{-i \hat{\theta}}\Omega_t  = 0,
  \]
holds since this integral is equal to the integral of $ f^* {\rm Im} \, e^{-i \hat{\theta}} \Omega_t$ by a deformation argument \cite{Marshall}, and the latter is zero by the choice of the angle (\ref{eq:angle-cond}).
  \end{proof}

 We now compute the linearization of $\F$ at the origin. The family of maps $\psi_\epsilon:= \varphi[\epsilon \alpha]$ satisfies $\psi_0=id$ and ${d \over d \epsilon} |_{\epsilon=0} \, \psi_\epsilon = \xi$ with $\xi= J f_* h^{-1} \alpha$. Thus
  \[
D \F|_0 (\alpha) = {d \over d \epsilon} \bigg|_{\epsilon=0} \F(\epsilon \alpha) = -f^* L_\xi \omega_t - \star_{h_{t}} f^* L_\xi {\rm Im} \, e^{-i \hat{\theta}} \Omega_t.
\]
Working on the local set $V_t$ with $t>0$, we introduce the model 2-form $\omega_{co,t}$ and model 3-form $\Omega_{mod,t}$ into the expression and rewrite this as
 \[
   D \F|_0(\alpha) = -f^* (L_\xi \omega_{co,t}) - \star_{h_{t}} f^* L_\xi {\rm Im} \, \Omega_{mod,t} + \mathcal{E}
 \]
 where
 \bea \label{eq:linearizederror}
 \mathcal{E} &=&  (-{\rm Re} \, e^{-i \hat{\theta}} +1)\star_{h_{t}} f^* L_\xi {\rm Im} \, \Omega_{mod,t} -  ({\rm Im} \, e^{-i \hat{\theta}}) \star_{h_{t}} f^* L_\xi {\rm Re} \, \Omega_{mod,t} \nonumber\\
 &&- f^*L_\xi (\omega_t - \omega_{co,t}) + {\rm Im}  \star_{h_t} f^* L_\xi (- e^{-i \hat{\theta}}\Omega_t +  e^{-i \hat{\theta}}\Omega_{mod,t}).
 \eea
Since $f^* \omega_{co,t} = f^* {\rm Im} \, \Omega_{mod,t}= 0$, $d \omega_{co,t}=0$ and $| \Omega_{mod,t}|_{g_{co,t}}=1$, the well-known \cite{McLean} formula for deformation of special Lagrangian submanifolds (\ref{eq:variation-formula}) gives
  \be \label{eq:linearized}
D \F|_0(\alpha) = (d +d^\dagger_{h_{t}}) \alpha + \mathcal{E}.
  \ee
The linearized operator is approximately $(d+d^\dagger)$ if the error term $\mathcal{E}$ is small. Motivated by this, we decompose
\[
\F(\alpha) = \F(0) + (d+d^\dagger_{h_{t}}) \alpha + \Q(\alpha),
\]
where we define
\[
\Q: \Lambda^1(S^3) \rightarrow {\rm Im}(d+d^\dagger)
\]
by
\[
\Q(\alpha) = \F(\alpha) - \F(0) - (d+d^\dagger_{h_{t}}) \alpha.
\]
Lemma \ref{lem:rangeF} gives the justification for $\Q(\alpha) \in d+d^\dagger$. We also define
\[
\mathcal{N}(\alpha) = (d+d^\dagger_{h_{t}})^{-1}[-\F(0) - \Q(\alpha)]
\]
so that $\F(\alpha)=0$ is equivalent to the fixed-point equation
\[
\mathcal{N}(\alpha)=\alpha.
\]

\subsection{Inverting $d+d^\dagger_{h_{t}}$} Let $\alpha \in \Lambda^*(S^3)$. We will use the
usual H\"older norms
\[
  \| \alpha \|_{C^\gamma(h)} = \| \alpha \|_{L^\infty(h)} + [ \alpha ]_{C^\gamma(h)}, \quad [\alpha]_{C^\gamma(h)} = \sup_{p \neq q} { |\alpha(p)-\alpha(q)|_h \over d_h^\gamma(p,q)}
\]
with the difference $(\alpha(p)-\alpha(q))$ understood by parallel transport along geodesics. Let
\[
W_t = d \Lambda^*(S^3) \oplus d^\dagger_{h_{t}} \Lambda^*(S^3).
\]
By standard elliptic theory (e.g. \cite{DougNirenberg,Marshall}), for $\gamma \in (0,1)$ we have the estimate
\[
\| \alpha \|_{L^\infty(h_1)} + \| \nabla_{h_1} \alpha \|_{C^\gamma(h_1)} \leq C \|(d+d^\dagger_{h_1}) \alpha \|_{C^\gamma(h_1)},
  \]
for all $\alpha \in W_1$. By the scaling relation $h_{t} = |t|^{2/3} h_1$, there exists $C>1$ independent of $t$ such that
\bea 
& \ & \| \alpha \|_{L^\infty(h_{t})} + |t|^{1/3} \| \nabla \alpha \|_{L^\infty(h_{t})} + |t|^{1/3}|t|^{\gamma/3} [ \nabla \alpha ]_{C^\gamma(h_{t})} \nonumber\\
&\leq& C |t|^{1/3} \bigg( \| (d+d^\dagger_{h_{t}}) \alpha \|_{L^\infty(h_{t})} + |t|^{\gamma/3}[ (d+d^\dagger_{h_{t}}) \alpha ]_{C^\gamma(h_{t})} \bigg)
\eea
for all $\alpha \in W_t$. As suggested by this estimate, we will use the weighted norms
\[
\| \alpha \|_{C^{0,\gamma}_t} := \| \alpha \|_{L^\infty(h_{t})} + |t|^{\gamma/3} [\alpha ]_{C^\gamma(h_{t})} 
\]
and
\[
\| \alpha \|_{C^{1,\gamma}_t} := \| \alpha \|_{L^\infty(h_{t})} + |t|^{1/3} \| \nabla \alpha \|_{L^\infty(h_{t})} + |t|^{1/3}|t|^{\gamma/3} [ \nabla \alpha ]_{C^\gamma(h_{t})} .
\]
The estimate on the inverse of $d+d^\dagger$ then becomes
\be \label{invert-linear}
\| \alpha \|_{C^{1,\gamma}_t} \leq C |t|^{1/3} \| (d+d^\dagger_{h_t}) \alpha \|_{C^{0,\gamma}_t}
\ee
for all $\alpha \in W_t$. For future use, we note the scaled inequality
\be \label{eq:holder2grad}
|t|^{\gamma/3} [\alpha ]_{C^\gamma(h_t)} \leq C |t|^{1/3} |\nabla \alpha|_{L^\infty(h_t)}
\ee
which follows from the inequality when $t=1$. 

\subsection{Contraction mapping}
We will work on the subspace of differential forms on $S^3$ given by $W_t = d \oplus d^\dagger_{h_{t}}$.  Let $0<\delta<(1/2)$, $0<\gamma<1$, and define
\be \label{eq:U-defn}
\mathcal{U}_t = \{ \alpha \in C^{1,\gamma}(W_t) :  \| \alpha \|_{C^{1,\gamma}_t} <  |t|^{1/3} |t|^\delta \}.
\ee
By the previous discussion, we have $\mathcal{N} : C^{1,\gamma}(W_t) \rightarrow C^{1,\gamma}(W_t)$, and now we aim to show that
\[
\mathcal{N}: \mathcal{U}_t \rightarrow \mathcal{U}_t
\]
is a contraction map. This will allow us to conclude the existence of a fixed point of $\mathcal{N}$. We start by showing the following estimate
\be \label{eq:contraction}
\| \mathcal{N}(u) - \mathcal{N}(v) \|_{C^{1,\gamma}_t} \leq {1 \over 2} \| u - v \|_{C^{1,\gamma}_t}, \quad u,v \in \mathcal{U}_t.
\ee
By the estimate (\ref{invert-linear}) on the inverse of $d+d^\dagger$, we estimate
\bea
\| \mathcal{N}(u) - \mathcal{N}(v) \|_{C^{1,\gamma}_t}  &=&  \| (d+d^\dagger_h)^{-1}(\mathcal{Q}(u) - \mathcal{Q}(v)) \|_{C^{1,\gamma}_t}  \nonumber\\
&\leq& C  |t|^{1/3} \| \mathcal{Q}(u) - \mathcal{Q}(v) \|_{C^{0,\gamma}_t}.
\eea
To prove (\ref{eq:contraction}), we will show
\be \label{Q-diff}
|t|^{1/3} \| \mathcal{Q}(u) - \mathcal{Q}(v) \|_{C^{0,\gamma}_t} \leq C |t|^{\delta} \|u-v\|_{C^{1,\gamma}_t}.
\ee
The difference between the $\mathcal{Q}$-terms is
\bea
\mathcal{Q}(u)-\mathcal{Q}(v) &=& \int_0^1 {d \over ds} \F(su + (1-s)v) \, ds  - (d+d^\dagger_{h_{t}}) (u-v) \nonumber\\
&=& \int_0^1 \bigg(D \F|_{u_s} - (d+d^\dagger) \bigg) (u-v) ds 
\eea
with $u_s = su+(1-s)v$. Therefore to show (\ref{Q-diff}), it suffices to estimate the difference between the approximate
linearized operator and the actual linearized operator. For $t$ small enough, estimate (\ref{eq:contraction}) thus follows from the following lemma.
\begin{lem}
\be \label{linearized-diff}
|t|^{1/3} \bigg\| \left( D \F|_{\eta} - (d+d^\dagger_{h_{t}})\right) \beta \bigg\|_{C^{0,\gamma}_t}  \leq C |t|^{\delta} \| \beta \|_{C^{1,\gamma}_t} 
\ee
for all $\eta \in \mathcal{U}_t$ and $\beta \in C^{1,\gamma}(W_t)$.
\end{lem}
\begin{proof} We can split the terms as
\[
\bigg( D \F|_\eta - D \F|_0 \bigg) + \bigg( D \F|_0-(d+d^\dagger_{h_{t}}) \bigg) := ({\rm I}) + ({\rm II})
\]
and we estimate these separately.
\smallskip
\par $\bullet$ Term $({\rm I})$. See e.g. \cite{Butscher,JoyceIII,Lee} for related calculations in the symplectic setting. We start by recalling the notation: let $t>0$, denote by $\omega_t, \Omega_t$ the global forms on $X_t$, $f_t: S^3 \rightarrow V_t \subset X_t$ the parametrization of the vanishing 3-sphere $\{ \|z \|^2 = |t| \}$, and $h_t = f_t^* g_{co,t}$ induced on $S^3$ from the model metric $g_{co,t}$. Given $\alpha \in \Lambda^1(S^3)$, the task is to understand the $t$-dependence of
\[
\F[\alpha] = \bigg[ - \varphi_t[\alpha]^* \omega_t - \star_{h_t} \varphi_t[\alpha]^* {\rm Im} \, e^{-i \hat{\theta}} \Omega_t \bigg] \in \Lambda^2(S^3) + \Lambda^0(S^3)
\]
where $\varphi_t[\alpha]: S^3 \rightarrow V_t$ is
\[
  \varphi_t[\alpha](u) = \exp_{g_{co,t}} (f_t(u), \xi_t), \quad \xi_t = J_t (f_t)_* h_t^{-1} \alpha.
\]
The scaling map $S: V_t \rightarrow V_1$ given by $S(z) = t^{-1/2} z$ satisfies $S^*(|t|^{2/3} g_{co,1}) = g_{co,t}$. By naturality of the exponential map,
\bea
S \circ \varphi_t[\alpha] (u) &=& \exp_{|t|^{2/3} g_{co,1}} (f_1(u), S_* \xi_t) \nonumber\\
&=& \exp_{g_{co,1}}(f_1(u), |t|^{-2/3} \xi_1) \nonumber\\
&=&  \varphi_1 \left[|t|^{-2/3} \alpha \right](u) .
\eea
Here we used $f_t = S^{-1} \circ f_1$ and $h_t = |t|^{2/3} h_1$. We now just write $\varphi=\varphi_1$. We can rewrite
\[
\varphi_t[\alpha]^* \omega_t = (S \circ \varphi_t[\alpha])^*
(S^{-1})^* \omega_t = |t|^{2/3} \varphi[|t|^{-2/3}\alpha]^* \check{\omega}_t
\]
where $\check{\omega}_t = |t|^{-2/3} (S^{-1})^* \omega_t$ is a
sequence of metrics on $V_1$ which is uniformly bounded with respect
to $g_{co,1}$; boundedness follow from pulling-back (\ref{eq:cot2check}). Similarly
\[
\star_{h_t} \varphi_t[\alpha]^* {\rm Im} \, e^{- i\hat{\theta}} \Omega_t = \star_{h_1} \varphi[|t|^{-2/3}\alpha]^* \check{\Omega}_t,\quad \check{\Omega}_t = |t|^{-1} (S^{-1})^* \Omega_t
\]
for $\check{\Omega}_t$ a sequence uniformly bounded in $t$ on $(V_1,g_{co,1})$. Putting everything together, we can see the $t$-dependence of $\F[\alpha]$ in local coordinates $\{ u^i \}$ on $S^3$ and coordinates $\{ x^\mu \}$ on $V_1$ and geometric objects on the fixed $V_1$. We write
\[
\F[\eta](u) = P \left[ |t|^{-2/3}\eta \right](u)
\]
where the map $P[\alpha]$ in $u^i,x^\mu$ coordinates is
\[
P[\alpha] = {1 \over 2} |t|^{2/3} P_{ij} du^i \wedge du^j + \hat{P}
\]
with
\[
P_{ij} = -\bigg[{\partial \varphi^\mu \over \partial u^i} {\partial \varphi^\nu \over \partial u^j} (\check{\omega}_t)_{\mu \nu}\circ \varphi \bigg] 
\]
\[
\hat{P} = -\star_{h_1} \bigg[ {\partial \varphi^\mu \over \partial u^i} {\partial \varphi^\nu \over \partial u^j} {\partial \varphi^\gamma \over \partial u^k} {\rm Im} \, e^{-i \hat{\theta}} ( \check{\Omega}_t)_{\mu \nu \gamma} \circ \varphi \bigg].
\]
Note that $\varphi$ only depends on the geometry of $(V_1,g_{co,1})$ and the map $P[\alpha]$ in local coordinates depends on $(\alpha_i, \partial_j \alpha_k)$ which we write as $P(\alpha_i,p_{jk})$.
\par We now compute the linearization of $\F = {1 \over 2} F_{ij} du^i \wedge d u^j + \hat{F}$. We start with the 2-form part $F_{ij}[\eta] = |t|^{2/3} P_{ij}[|t^{-2/3}\eta]$. The derivative at $\eta \in \mathcal{U}$ in the direction $\beta \in W$ is
\[
[DF|_\eta\beta]_{ij} = \bigg( {\partial P_{ij} \over \partial \alpha_k} \left[ |t|^{-2/3} \eta \right] \beta_k + {\partial P_{ij} \over \partial p_{k \ell}} \left[ |t|^{-2/3}\eta \right] \partial_k \beta_\ell \bigg).
\]
By the mean value theorem, there is $\eta_\theta = \theta |t|^{-2/3} \eta$, $0 \leq \theta \leq 1$ such that
\bea
[DF|_\eta\beta - DF|_0 \beta]_{ij} &=& |t|^{-2/3} \bigg[{\partial^2 P_{ij} \over \partial \alpha_\ell \partial \alpha_k}(\eta_\theta) \eta_\ell \beta_k + {\partial^2 P_{ij} \over \partial p_{rs} \partial \alpha_k}(\eta_\theta) \partial_r \eta_s \beta_k \nonumber\\
&&+ {\partial^2 P_{ij} \over \partial \alpha_p \partial p_{k \ell}}(\eta_\theta) \eta_p \partial_k \beta_\ell + {\partial^2 P_{ij} \over \partial p_{rs} \partial p_{k \ell}}(\eta_\theta) \partial_r \eta_s \partial_k \beta_\ell \bigg].
\eea
Let $V = \{\alpha \in \Lambda^1(S^3) : |\alpha|_{f_1^* g_{co,1}} < 1 \}$ so that $\eta \in \mathcal{U}_t$ implies $\eta_\theta \in V$ and $|\xi(\eta_\theta)|_{g_{co,1}} < 1$, and all tensors in the definition of $P$ are restricted to a compact set containing the 3-cycle $L_1=\{ \| z \|^2 = 1 \}$ in $(V_1,g_{co,1})$. We now take the norm of this 2-form with respect to $h_t = |t|^{2/3} h_1$.
\bea
& \ &  \bigg| (( DF|_\eta - DF|_0) \beta)_{ij}
\bigg|_{h_t}   \nonumber\\
&\leq& |t|^{-4/3} \bigg[ \sup_{\alpha \in V} \bigg|{\partial^2 P_{ij} \over \partial \alpha_\ell \partial \alpha_k} \bigg|_{h_1} |\eta |_{h_1} |\beta|_{h_1} + \sup_{\alpha \in V}  \bigg|  {\partial^2 P_{ij} \over \partial p_{k \ell} \partial p_{rs}} \bigg|_{h_1} |\nabla \eta|_{h_1} |\nabla \beta|_{h_1} \nonumber\nonumber\\
&&+ \sup_{\alpha \in V} \bigg|  {\partial^2 P_{ij} \over \partial \alpha_p \partial p_{k \ell}} \bigg|_{h_1} (|\nabla \eta|_{h_1} |\beta|_{h_1} + |\eta|_{h_1} |\nabla \beta|_{h_1} ) \bigg]\nonumber\\
&\leq& C |t|^{-4/3} (|\eta|_{h_1} |\beta|_{h_1} + |\nabla \eta|_{h_1} |\nabla \beta|_{h_1} + |\nabla \eta|_{h_1} |\beta|_{h_1} + |\eta|_{h_1} |\nabla \beta|_{h_1} ) \nonumber\\
&\leq&  C  |t|^{-4/3} |t|^{2/3} \|\eta \|_{C^{1,\gamma}_t} \| \beta \|_{C^{1,\gamma}_t}. \nonumber
\eea
This is the 2-form contribution $F_{ij}$ coming from $\varphi_t^* \omega_t$. The 0-form contribution $\hat{F}$ coming from $\star_t \varphi_t^* {\rm Im} \, e^{-i \hat{\theta}} \Omega_t$ is estimated in the same way:
\[
\bigg| (D \hat{F}|_\eta - D \hat{F}|_0) \beta \bigg| \leq C |t|^{-4/3} |t|^{2/3} \|\eta \|_{C^{1,\gamma}_t} \| \beta \|_{C^{1,\gamma}_t}.
\]
Therefore, using $\| \eta \|_{C^{1,\gamma}_t} < |t|^{1/3} |t|^\delta$ and multiplying through by $|t|^{1/3}$, we obtain
\[
|t|^{1/3} \bigg| \bigg( D \F|_\eta - D \F|_0 \bigg) \beta
\bigg|_{h_t}  \leq C  |t|^\delta \| \beta \|_{C^{1,\gamma}_t}.
\]
This estimate is of the desired form (\ref{linearized-diff}) for the contribution of term $({\rm I})$. The estimate of the H\"older norm
\[
  |t|^{1/3} |t|^{\gamma/3} \bigg| \bigg( D \F|_\eta - D \F|_0 \bigg) \beta \bigg|_{C^\gamma(h_t)} \leq C  |t|^\delta \| \beta \|_{C^{1,\gamma}_t}
\]
can be proved by a similar scaling argument.

\smallskip
\par $\bullet$ Term ${\rm II}$. We use the notation in (\ref{eq:linearized}). We must estimate
\be
\bigg| \bigg( D \F|_0-(d+d^\dagger) \bigg) \beta \bigg|_{C^{0,\gamma}_t} = |\mathcal{E}(\beta)|_{h_t} + |t|^{\gamma/3} [\mathcal{E}(\beta)]_{C^\gamma(h_t)}
\ee
where the error terms $\mathcal{E}$ are given in (\ref{eq:linearizederror}). Schematically, the error term is of the form
\bea
& \ & \mathcal{E}(\beta) \nonumber\\
&=&  \left( {\rm Re} \, e^{-i \hat{\theta}} - 1 \right)  f_t^* (J_t* {\rm Im} \, \Omega_{mod,t}) * \nabla_{h_t} \beta  +\left(  {\rm Im} \, e^{-i \hat{\theta}} \right)  f_t^*(J_t* {\rm Re} \, \Omega_{mod,t} )* \nabla_{h_t} \beta  \nonumber\\
&&+ \nabla_{h_t} \beta * f_t^*(g_t - g_{co,t}) + \beta * \nabla_{h_t} f_t^* (g_t - g_{co,t}) + \beta * f_t^*(J_t * \nabla_{g_{co,t}} g_t) \nonumber\\
&&+ \nabla_{h_t} \beta * {\rm Im} \,  f_t^* (e^{-i \hat{\theta}}\Omega_t - e^{-i \hat{\theta}}\Omega_{mod,t} ) + \beta * {\rm Im} \, \nabla_{h_t} f_t^* ( e^{-i \hat{\theta}} \Omega_t -  e^{-i \hat{\theta}} \Omega_{mod,t} ) \nonumber
\eea
where $*$ denotes contractions where indices may be raised or lowered using $h_t$ (or $g_{co,t}$ if before pullback by $f_t$).  Here we have used again that $\Omega_{mod,t}$ is parallel with respect to $g_{co,t}$.   Using (\ref{eq:cot2check}), (\ref{eq:Omegamod2Omega}), and (\ref{eq:zeta-est}), this can be estimated by
\[
|t|^{1/3} | \mathcal{E}(\beta) |_{h_t} \leq C |t|^{1/2} \| \beta \|_{C^{1,\gamma}_t}.
\]
The H\"older norms $|t|^{1/3}|t|^{\gamma/3} | \mathcal{E}(\beta) |_{C^\gamma(h_t)}$ can be estimated similarly, using that the H\"older norms on the background terms $f_t^*(\omega_t-\omega_{co,t})$, $f_t^*(\Omega_t-\Omega_{mod,t})$ can be converted to gradient norms using (\ref{eq:holder2grad}), and then estimates (\ref{eq:cot2check}), (\ref{eq:Omegamod2Omega}) can be used. This completes the proof of the estimate (\ref{linearized-diff}).
\end{proof}

\par We have thus proved the contraction mapping estimate (\ref{eq:contraction}) for $t$ small enough. We now check whether $\mathcal{U}_t$ is preserved by $\mathcal{N}$. Let $u \in \mathcal{U}_t$. Then
\bea
\| \mathcal{N}(u) \|_{C^{1,\gamma}_t} &\leq& \| \mathcal{N}( 0) \|_{C^{1,\gamma}_t} + \| \mathcal{N}(0) - \mathcal{N}(u) \|_{C^{1,\gamma}_t} \nonumber\\
&\leq& \| \mathcal{N}(0) \|_{C^{1,\gamma}_t} + {1 \over 2} \| u \|_{C^{1,\gamma}_t}
\eea
by (\ref{eq:contraction}). Using the estimate (\ref{invert-linear}) on the inverse of $d+d^\dagger$, we estimate
\[
\| \mathcal{N}(0) \|_{C^{1,\gamma}_t} \leq C |t|^{1/3} \| \F(0) \|_{C^{0,\gamma}_t}.
\]
By Lemma \ref{lem:approx-soln}, the approximate special 3-cycle $f:L \rightarrow X_t$ satisfies
\[
\| \F(0) \|_{C^{0,\gamma}_t} \leq \| f^* \omega_t \|_{C^{0,\gamma}_t} + \| f^* {\rm Im} \, e^{-i \hat{\theta}} \Omega_t \|_{C^{0,\gamma}_t} \leq C |t|^{1/2}. 
\]
Here the weighted $C^\gamma$ norms was converted to the weighted gradient norm (\ref{eq:holder2grad}). Altogether, we have that
\[
\| \mathcal{N}(u) \|_{C^{1,\gamma}_t} \leq C |t|^{1/3} |t|^{1/2} + {1 \over 2} |t|^{1/3} |t|^\delta <  |t|^{1/3} |t|^\delta
\]
for $t$ small and $0<\delta < (1/2)$. This shows that $\mathcal{N}$ preserves $\mathcal{U}_t$. Therefore $\mathcal{N} : \mathcal{U}_t \rightarrow \mathcal{U}_t$ is a contraction mapping. By the Banach fixed point theorem, $\mathcal{N}$ has a fixed point $\alpha \in \mathcal{U}_t$.
\smallskip
\par Our setup is such that the equation $\mathcal{N}(\alpha)=\alpha$ is equivalent to $\mathcal{F}(\alpha)=0$. We thus obtain a $C^{1,\gamma}$ solution of Theorem \ref{thm:3spheres}. The resulting submanifold $\tilde{L} \subset X_t$ is smooth by elliptic regularity theory for systems. This follows from Morrey's regularity \cite{Morrey} of minimal surfaces. Indeed, we noted in \S \ref{section:defn} that $\tilde{L}$ minimizes the area functional $A(L)=\int_L d {\rm vol}_{f^* \bar{g}}$ in a homology class $[\tilde{L}] \in H_3(X,\mathbb{R})$, where $\bar{g} = | \Omega |_{g_t}^{2/n} g_t$. Therefore $\tilde{L}$ is a minimal surface in $(X,\bar{g})$, and Morrey's regularity theorem states that $C^{1,\gamma}$ minimal surfaces are smooth. For a recent modification of Morrey's regularity theory to Hamiltonian stationary submanifolds in a symplectic manifold, see \cite{BhatChenWar}.
\smallskip
\par 
To conclude, we show that the non-K\"ahler special Lagrangian 3-sphere $\tilde{L}$ is rigid. Suppose that a solution $\alpha$ of the linearized special Lagrangian equation deforms to a family $\alpha_\epsilon$ of solutions $\F(\alpha_\epsilon)=0$ with ${d \over d \epsilon}\big|_{\epsilon=0} \alpha = \beta$. Then $D \F|_{\alpha} (\beta) = 0$. Since $H^1(S^3)=0$, then $\beta \in d\oplus d^\dagger$ and we may apply (\ref{invert-linear}):
\bea
\| \beta \|_{C^{1,\gamma}_t} &\leq& C |t|^{1/3} \| (d+d^\dagger_{h_t}) \beta \|_{C^{0,\gamma}_t} \nonumber\\
&\leq& C |t|^{1/3} \| D \F|_\alpha (\beta) \|_{C^{0,\gamma}_t} + C|t|^{1/3}\| (D\F|_\alpha - (d+d^\dagger_{h_t})) \beta \|_{C^{0,\gamma}_t} \nonumber\\
&\leq& C |t|^{1/3} \| D \F|_\alpha (\beta) \|_{C^{0,\gamma}_t} + C |t|^\delta \| \beta \|_{C^{1,\gamma}_t}
\eea
by (\ref{linearized-diff}). For $t$ small enough, we see that $D \F|_\alpha (\beta) = 0$ implies $\beta=0$. It follows that the special Lagrangian vanishing cycle on $X_t$ defined by $\alpha \in \mathcal{U}_t$ cannot be deformed inside $X_t$.

\section{Relation to $SU(3)$ structures and flux compactifications}\label{sec: phys}

The problem discussed in this paper has direct applications to string theory, in the context of flux compactifications and compactifications on a 6-manifold $X$ with $SU(3)$ structure. An $SU(3)$ structure $(\omega,\Psi)$ is given by a reduction of the structure group to $SU(3)$ with trivializing local co-tangent frames denoted $\{ e^i \}_{i=1}^6$, together with a 2-form $\omega$ and a 3-form $\Psi$ which appear in a local frame as
\begin{eqnarray}
  \omega &=& e^1 \wedge e^2 + e^3 \wedge e^4 + e^5 \wedge e^6, \notag \\
  \Psi &=& (e^1+ie^2) \wedge (e^3+ie^4) \wedge (e^5+ie^6). \notag
\end{eqnarray}
In general, on such a manifold $X$ the failure of $\omega$ and $\Psi$ to remain closed is parametrized by five torsion classes,
\begin{eqnarray}
d \omega & = & \frac{3}{2} \text{Im} (\bar W_1 \Psi) + W_4 \wedge \omega + W_3 \label{su3str} \\
d \Psi & = & W_1 \, \omega^2 + W_2 \wedge \omega + \bar W_5 \wedge \Psi \notag
\end{eqnarray}
which, in string compactification on $X$, are interpreted as NS-NS and R-R fluxes. In our setup of a complex manifold with trivial canonical bundle, we denoted by $\Omega$ the holomorphic volume form, and in this case the constant norm 3-form is $\Psi = \frac{\Omega}{|\Omega|}$. This means that a conformally balanced metric $\omega$ on $X$ has non-zero torsion classes $W_3$, $W_4$, and $W_5 = 2 W_4$, while $W_1 = 0 = W_2$. This pattern of torsion classes can be realized in both type II and the heterotic string theory, by using only the fields from the NS-NS sector. Namely, turning on the fluxes for the NS-NS 2-form field $B$ and for the dilaton, one can construct in this way 4d $\mathcal{N}=1$ supersymmetric vacua with zero cosmological constant, see {\it e.g.} \cite{AndersonKark,Grana,LLR}.

In the effective four-dimensional supergravity, these fluxes generate a superpotential $\mathcal{W} \sim \int \Omega \wedge (H + i d \omega)$ \cite{GVW,CCDL,dlOHS} that (partly) lifts moduli of the original system without fluxes (torsion). In type II string theory, these are Coulomb and Higgs branches, $\mathcal{M}_C$ and $\mathcal{M}_H$, parametrized by the fields in $\mathcal{N}=2$ vector and hypermultiplets. For example, in a 4d effective theory with $n_H$ hypermultiplets, the moduli space $\mathcal{M}_H$ is a quaternionic-K\"ahler space of real dimension $4 n_H$ and the Ricci scalar curvature $R = - 8 n_H (n_H + 2)$, whereas $\mathcal{M}_C$ is a special K\"ahler manifold. We have $(n_V,n_H) = (h^{1,1}, h^{2,1}+1)$ in type IIA string theory, and $(n_V,n_H) = (h^{2,1}, h^{1,1}+1)$ in type IIB string theory. Similarly, in the heterotic string without torsion, the moduli consist of the complex and K\"ahler structure deformations, as well as the moduli of the bundle $E \to X_t$. In the presence of fluxes (torsion), it was shown in the string theory literature \cite{AGS,dlOSvanes} that there is a map from the infinitesimal parameter space of the Strominger system to:
\be
H^{(0,1)} (T^* X_t) \; \oplus \; \text{ker} \big( H^1 (\mathcal{E}) \to H^{(0,2)} (T^* X_t) \big)
\ee
where $\mathcal{E}$ fits into the exact sequence $0 \to \text{End} (TX_t) \oplus \text{End} (E) \to \mathcal{E} \to TX_t \to 0$. In the mathematics literature, there is another approach to the moduli problem via the notion of a string algebroid and symmetries in generalized geometry; see \cite{GRT17,GRT,GRST}.

Now, using these tools, we are ready to describe the physical interpretation of the special Lagrangians in a topology changing transition to the conformally balanced non-Kahler metric on $X_t$. Since the topology changing transition \eqref{top-change} decreases $b_2$ and increases $b_3$, in type IIA string theory it corresponds to a phase transition in the effective 4d field theory from the Coulomb branch to the Higgs branch. This is the usual Higgs mechanism, in which charged particles condense breaking (part of) the gauge symmetry. The reverse process is the transition from Higgs to Coulomb branch, which is also the interpretation of \eqref{top-change} in type IIB string theory.

Following \cite{Stromingerholes,GMS}, we note that D2-branes in type IIA theory on $X$ are massive charged particles, charged under the gauge group $U(1)^{b_2 (X)}$ of the effective 4d theory. Due to the Friedman's condition \eqref{Friedmans-cond}, their charges are not independent. As a result, contracting $k$ disjoint $(-1,-1)$ curves $C_i \subset X$ which span a subspace of dimension $(k-1)$ in $H_2(X,\mathbb{R})$, has the effect of decreasing the number of vector multiplets by $k-1$ and increasing the number of hypermultiplets by $1$. Therefore, we expect $b_2 (X_t) = b_2 (X) - k + 1$ and $b_3 (X_t) = b_3 (X) + 2$.

The vanishing cycles $L_i$, $i = 1, \ldots, k$, have a more natural interpretation in type IIB string theory. Namely, D3-branes on $L_i$ can be interpreted as massive particles charged under $U(1)^{b_3 (X_t)/2}$ gauge group. Their condensation leads to the reverse of the transition \eqref{top-change}. However, these charged particles constructed from D3-branes in type IIB theory are not BPS. Indeed, in the context of the non-K\"ahler metric on $X_t$, the effective 4d theory has only $\mathcal{N}=1$ supersymmetry, which admits BPS strings but not BPS particles. The BPS strings can be produced by wrapping D3-branes on $(-1,-1)$ curves $C_i \subset X$. They are charged under 2-form tensor fields dual to compact periodic scalars in the hypermultiplets.

Finally, we remark that a transition \eqref{top-change} from a K\"ahler metric on $X$ to a non-K\"ahler metric on $X_t$ corresponds to a phase transition in the effective 4d supergravity theory from a branch with $\mathcal{N}=2$ supersymmetry to a branch with $\mathcal{N}=1$ supersymmetry. We propose the physical interpretation of such transitions in terms of soft SUSY-breaking terms that give a small mass to a chiral superfield $\Phi$ either inside a vector or inside a hypermultiplet. Then, on the branch where the complete $\mathcal{N}=2$ multiplet is massive, the small soft SUSY-breaking parameters are negligible and $\mathcal{N}=2$ supersymmetry is preserved, especially far out on this branch where loop and other quantum effects are suppressed. However, on the other branch --- where the $\mathcal{N}=2$ multiplet that contains $\Phi$ is massless --- soft SUSY-breaking terms produce the leading effect.

\end{document}